%% file: main_revised.tex
\documentclass[11pt]{article}

\input{head/package}

\input{head/thm}
\input{head/macro}
\usepackage{amsmath}
\usepackage{comment}

\usepackage{CJKutf8}

\usepackage{datetime}
\date{\vspace{-2.5\baselineskip}}

\author[1]{Yuya Hikima\footnote{Corresponding author. E-mail: \url{yuya.hikima@ntt.com}}}
\author[2]{Yasunori Akagi}
\author[2]{Hideaki Kim}

\affil[1]{NTT Communication Science Laboratories, NTT Corporation,
Kyoto, Japan}
\affil[2]{NTT Human Informatics Laboratories, NTT Corporation,
Kanagawa, Japan}

\title{Joint Pricing and Matching for Resource Allocation Platforms via Min-cost Flow Problem}

\begin{document}
\maketitle

\begin{abstract}
Stochastic matching is the stochastic version of the well-known matching problem, which consists in maximizing the rewards of a matching under a set of probability distributions associated with the nodes and edges. In most stochastic matching problems, the probability distributions inherent in the nodes and edges are set a priori and are not controllable. However, many resource allocation platforms can control the probability distributions by changing prices. For example, a rideshare platform can control the distribution of the number of requesters by setting the fare to maximize the reward of a taxi-requester matching. Although several methods for optimizing price have been developed, optimizations in consideration of the matching problem are still in its infancy. In this paper, we tackle the problem of optimizing price in the consideration of the resulting bipartite graph matching, given the effect of the price on the probabilistic uncertainty in the graph. Even though our problem involves hard to evaluate objective values and is non-convex, we construct a $(1-1/e)$-approximation algorithm under the assumption that a convex min-cost flow problem can be solved exactly.
\end{abstract}

\section{Introduction} 
\label{sec: intro}
Bipartite graph matching has received a great deal of attention as a fundamental discrete optimization problem that has wide-ranging applications, such as matching of workers to jobs, residents to hospitals, and jobs to machines in cloud computing \citep{ahuja1988network}. Among the variants of bipartite graph matching, stochastic matching techniques have recently been developed to handle probabilistic uncertainties of instances: they determine the optimum matching under a given probability distribution associated with the nodes and edges. In kidney exchange \citep{chen2009approximating}, for example, the expected number of patient-donor matchings is maximized under a given set of transplantability probabilities between patients and donors. Moreover, in internet advertising \citep{mehta2013online}, advertisements are allocated to website visitors based on probability distributions over types of visitor, which are estimated from past traffic data from websites.

In conventional stochastic matching problems, the probability distributions inherent in the nodes and edges are assumed to be given a priori and so are not controllable. However, many resource allocation platforms can control the probability distributions by changing \textit{price}. For example, a crowd-sourcing servicer can manage worker participation rates for a task by increasing or decreasing the \textit{payment} \citep{horton2010labor}, while a rideshare platform provider can control the distribution of the number of requesters in each area by setting the \textit{fare} \citep{Tong2018DynamicPI}. Since the change in the probability distributions inherent in the graphs affects the resulting optimum matching, finding the price that maximizes the business profit is one of the central concerns for service providers. For example, in the case of a crowd-sourcing platform, the provider can eliminate a surplus of difficult tasks by attracting highly skilled workers with higher wages and allocating the tasks to them. 

In this paper, we focus on the problem of optimizing the price under the consideration of the resulting bipartite graph matching, given the effect of price on the probabilistic uncertainty in the graph. Several studies have already engaged this problem. For example, \cite{Tong2018DynamicPI} proposed a method to optimize prices to maximize the expected weight of the resulting matching. \cite{chen2019dispatching} proposed a contextual bandit algorithm to learn a profitable price setting for each taxi request and optimize the taxi-requester matching by solving a bipartite matching problem. Although these methods have proven successful in pricing, there is room for their growth in two important aspects. First, the tackled problems impose restrictions on the graph structure and price: the method of \citep{Tong2018DynamicPI} can only deal with vertex-weighted graphs, not edge-weighted graphs, which makes it impossible to consider worker-task compatibility in crowd-sourcing and driver-requester distance in ride-sharing; the method of \citep{chen2019dispatching} limits the price that can be set to a predetermined discrete set of alternatives, which causes coarse pricing. Second, their method has no theoretical guarantees regarding the objective value to be achieved, or if it does, the approximation rate deteriorates as the problem size increases. This may degrade the performance of the method especially when the problem size is large.

To overcome these issues, we propose a generalized problem from that of \citep[Section 2.2]{Tong2018DynamicPI}: while \cite{Tong2018DynamicPI} assume vertex-weighted graphs, our problem assumes edge-weighted graphs. This generalization allows us to consider worker-task compatibility in crowd-sourcing and driver-requester distance in ride-sharing. Moreover, in contrast to the work of \citep{chen2019dispatching}, our problem treats price as a continuous variable and does not limit the range of prices to a discrete set of a priori candidates. This leads to more appropriate price setting.

Moreover, we propose a $(1-1/e)$-approximation algorithm under the assumption that a convex minimum-cost (min-cost) flow problem can be solved exactly. Here, we had to tackle two difficulties: (i) the objective function is hard to evaluate because it is the expectation of optimal values of b-matching problems with uncertainty; (ii) it is a non-convex problem. To overcome difficulty (i), we approximated the objective value and reduced our problem to a non-linear cost flow problem. The objective value of the reduced problem could then be easily calculated and the optimal solution would be a $(1-1/e)$-approximation solution for our problem. To deal with difficulty (ii), we showed that the reduced problem is a convex min-cost flow problem under the assumption that the mean of the number of stochastic nodes follows a log-concave function \citep{bagnoli2006log,barlow1963properties}) with respect to price. This assumption is reasonable because the log-concave function category includes many important functions such as the Gauss error function and logistic function, which are often used in the price optimization literature \citep{carvalho2005learning,dong2009dynamic,schulte2020price}. Moreover, convex min-cost flow problems can be solved efficiently with a number of algorithms \citep{ahuja1988network,vegh2016strongly}.

We conducted simulation experiments using real data from ride-sharing and crowd-sourcing platforms. The results show that the proposed algorithm outputs solutions with higher objective values than those of the conventional methods in a practical amount of time.

The contributions of this study are twofold.
\begin{itemize}
\item We formulate a new optimization problem for joint pricing and matching. Our problem is a generalization of \citep{Tong2018DynamicPI} and can handle more practical situations in resource allocation platforms.
\item We propose a fast algorithm that achieves high objective values for our problem. The algorithm achieves a $(1-1/e)$-approximation solution under the assumption that the convex min-cost flow problem can be solved exactly.
\end{itemize}

This paper is an extension of \citep{hikima2021integrated}. The differences are summarized as follows. First, we extend the problem setting so that the probability distributions for stochastic nodes can handle binomial and Poisson distributions in addition to Bernoulli distributions (Section \ref{sec:matching_procedure}), which enables us to set {\it posted prices} (prices before customers appear) while the previous problem setting \citep{hikima2021integrated} allows us to quote prices only after customers have been gathered. This eliminates the overhead time needed to quote prices to customers (Details are in the footnote in Section \ref{sec:matching_procedure}). Second, we extended the node capacity of the problem: each node can have multiple capacities in our setting, whereas it has only one capacity in \citep{hikima2021integrated}. Although \citep{hikima2021integrated} can manage situations where each node has multiple capacities by replicating the same number of nodes as capacities, this enlarges the problem and leads to a long computation time. Besides these extensions, we updated a proof about the bounds of the objective function (Theorem \ref{thm:approximation_ratio} in Section \ref{Aof}). Third, we augmented the theoretical analyses by adding Lemma~\ref{lem:opt_exist} (Section \ref{Aof}). Lemma \ref{lem:opt_exist} shows that the transformed problem with the approximated objective function, (PA) in Section \ref{Aof}, has an optimal solution ensuring that our algorithm always works correctly under the assumptions.

\paragraph{Notation} Bold lowercase symbols (e.g., $\bx, \by$) denote vectors. $\mR_{\ge0}$ and $\Z_{\ge0}$ are the set of positive real numbers and positive integers, respectively. The derivative for a real-valued function $p$ is denoted by $p'$.

\section{Related work} 
\label{sec:related_works}
\subsection{Pricing and Matching in Resource Allocation Platforms}
\label{sec:related_work_joint_pricing_matching}
Many studies have explored matching and pricing methods, which can be divided into three categories: (i) matching methods without optimization of price decisions; (ii) pricing methods without optimization of matching decisions; (iii) joint pricing and matching methods.

The existing studies on (i) \citep{hu2022dynamic,zhao2019preference,tong2016online,ma2013t,zheng2018order} focus on matching decisions under some pricing rule.\footnote{In some cases, they do not mention the pricing rule.} For example, \cite{hu2022dynamic} decide requester-taxi matchings in ride-hailing platforms by ignoring pricing decisions, whereas \cite{zhao2019preference} and \cite{tran2014efficient} decide worker-task matchings without considering task pricing in crowd-sourcing.

The existing studies on (ii) \citep{Xiao2020OptimalCC,mao2013pricing,Besbes2021surge,bimpikis2019spatial,castillo2017surge,guda2019your,banerjee2015pricing,singer2013pricing} focus on pricing decisions under some matching rule. For example, \cite{banerjee2015pricing, Besbes2021surge, bimpikis2019spatial}, and \cite{guda2019your} decide the price for each taxi request in a ride-hailing platform by restricting the matches to being between requesters and taxis in the same area. Moreover, \cite{Xiao2020OptimalCC}, \cite{singer2013pricing}, and \cite{mao2013pricing} perform price optimization and prediction for crowd-sourcing but do not consider worker-task matching.

In regard to (iii), there have been some studies \citep{chen,Tong2018DynamicPI,shah2022joint,ma2022spatio,hikima2022online,Hikima_Akagi_Kim_Asami_2023,feng2023two} on simultaneously deciding pricings and matchings. Several among them \citep{Tong2018DynamicPI,shah2022joint,feng2023two} are similar to ours in terms of the optimization problem they tackle, i.e., two-stage stochastic problems, where the first stage is to optimize the price and the second stage is to optimize the matching. \cite{Tong2018DynamicPI} estimate the acceptance probabilities of requesters for price by using bandit techniques and tackle the two-stage problem. Although their work is similar to ours, ours has two distinct improvements. First, our algorithm outputs a $(1-1/e)$-approximation solution via an optimal solution of a convex min-cost flow problem, while previous methods do not offer constant approximation ratios. Second, our problem is a generalization of that of \citep{Tong2018DynamicPI}: while \cite{Tong2018DynamicPI} assume vertex-weighted graph for matching decisions, we allow edge-weighted graphs, which enables us to incorporate, e.g., worker-task compatibility in crowd-sourcing and driver-requester distance in ride-sharing as objectives. \cite{shah2022joint} proposed a heuristic method for pricing and matching in ride-hailing platforms where each vehicle has multiple capacities for different trips. In contrast to their research, we assume that nodes on one side of the bipartite graph (, which correspond to vehicles) have only unit capacity and our algorithm outputs a $(1-1/e)$-approximation solution via an optimal solution of a convex min-cost flow problem. 
\cite{feng2023two} formulate an optimization problem to maximize drivers' efficiency and requesters' utility,
whereas our objective is to maximize the platform's profit.

Other studies \citep{hikima2022online,Hikima_Akagi_Kim_Asami_2023,chen,ma2022spatio} have proposed joint pricing and matching methods without addressing the two-stage stochastic problem. \cite{hikima2022online,Hikima_Akagi_Kim_Asami_2023} tackled online matching situations and optimization problems for determining pricing and matching policies. In contrast to these studies, we assume batch-type matchings. In applications that tolerate some latency, we can make more profit by optimizing the matching to the accumulated demand \citep{uber2018,zhang2017taxi, feng2023two}. \cite{chen} used a contextual bandit algorithm to learn profitable price settings for each taxi request and optimized the taxi-requester matching by solving bipartite matching problems. Since this method sets the price of each request individually, it cannot take into account the balance between supply and demand, which possibly leads to a decrease in profits.
\cite{ma2022spatio} proposed spatio-temporal pricing mechanism, which achieves good properties for ride-share platforms such as welfare-optimality and envy-freeness.
It assumes that complete information on demand (taxi-requests) is available and thus cannot accommodate demand uncertainty.

\subsection{Optimization Methods for Stochastic Problems with Decision-dependent Uncertainty} 
\label{subsec:related_decision_dependent_noise}
Our problem can be regarded as a stochastic one with \textit{decision-dependent uncertainty} (which is also called \textit{endogenous noise}) \citep{hellemo2018decision,Varaiya1989} within the literature of mathematical optimization. This is because the existence of nodes follows a probability distribution depending on decision variables (price). We cite four methods below and show that they are not suitable for our problem.

\paragraph{Stochastic gradient descent methods \citep{hikima2023stochastic,sutton2018reinforcement}.} Gradient descent methods update the current iterate by using the stochastic gradient, such as $\bx_{k+1}:= \bx_k - \eta_k \bm{g}_k$, where $\bm{g}_k$ is an unbiased stochastic gradient (i.e., $\E[\bm{g}_k]=\nabla \mathbb{E}_{\bmxi \sim D(\bx)}[f(\bx,\bmxi)]|_{\bx=\bx^k}$). A projected stochastic gradient descent \citep{hikima2023stochastic} was proposed for pricing problems with decision-dependent demand and a Monte-Carlo policy-gradient method \citep[Section 13.3]{sutton2018reinforcement} was proposed for policy optimization problems in the literature of reinforcement learning. Although these methods converge to stationary points, they assume that the objective function is differentiable and are not applicable to our problem. Moreover, the approximation ratios of the solutions are not guaranteed.

\paragraph{Zeroth-order methods \citep{spall2005introduction,Flaxman2004online}.} Zeroth-order methods approximate the (sub-) gradient by evaluating objective values at randomly sampled points around the iterate. While this type of method can be used on our problem, it requires a long computation time to evaluate the objective values many times.

\paragraph{Retraining methods \citep{perdomo2020performative,mendler2020stochastic}.} Retraining methods update the current iterate by fixing the distribution, such as by setting $\bx_{k+1}:= \mathrm{proj}_{\mathcal C}(\bx_k - \eta_k \E_{\bmxi \sim D(\bx_k)} [\nabla_{\bx} f(\bx,\bmxi)|_{\bx=\bx_k}])$, where $\mathrm{proj}_{\mathcal C}$ is the Euclidean projection operator onto the feasible region $\mathcal C$. It converges to a \emph{performatively stable point} $\hat{\bx}$ such that $\hat{\bx} = \mathrm{arg} \min_{\bx} \E_{\bmxi \sim D(\hat{\bx})}[f(\bx,\bmxi)]$. However, these methods assume the smoothness and strong convexity of $f$ w.r.t. $\bx$ and are not applicable to our problem.

\paragraph{Bayesian optimization \citep{brochu2010tutorial,frazier2018tutorial}.} Bayesian optimization repeatedly finds better solutions by learning the objective function through Gaussian process regression. While this method is generic, it is not scalable and cannot find good solutions for problems such as ours, which is generally large-scale.

\section{Problem formulation \label{PF}}
\subsection{Pricing and matching procedure} 
\label{sec:matching_procedure}
This section describes the notation and matching procedure, which is illustrated in Figure \ref{over}.

\begin{figure}[t]
\centering
 \includegraphics[scale=0.5]{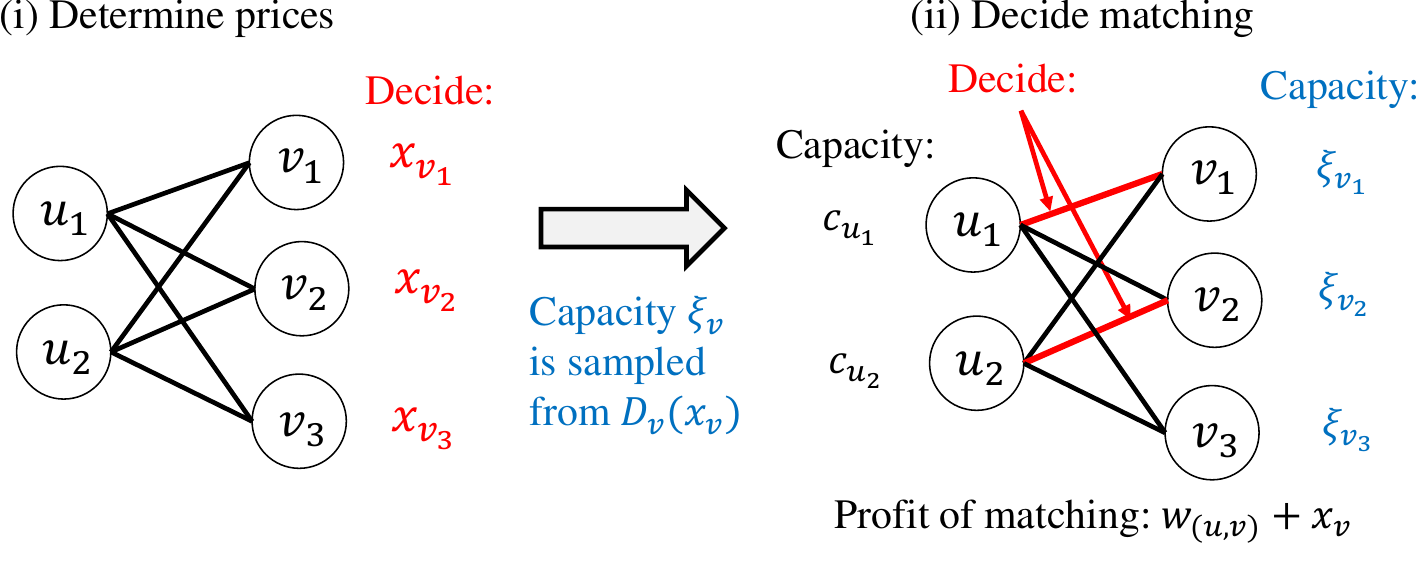}
\caption{Matching procedure}
\label{over}
\end{figure}

\paragraph{Notation.} We mathematically model the resource allocation platform by using a bipartite graph. The set of resource types $U$, the set of participant groups $V$, and a weighted bipartite graph $G=(U, V, E)$ are given; when $(u, v) \in E$ holds, a resource in type $u$ can be matched to a participant in group $v$. Each edge $e \in E$ is associated with a weight $w_e \in \R$. Each node of the bipartite graph has a \emph{capacity}, which is the number of times the node can be matched. Let $c_u \in \Z_{>0}$ denote the capacity of node $u \in U$; this value is given as input. We write $\xi_v \in \Z_{>0}$ for the capacity of the node $v \in V$; this is a random variable determined according to a certain probability distribution, as explained later. Note that $\xi_v$ is \emph{not} given as input. 

\paragraph{System procedure.} We consider the following pricing and matching procedure. First, the platform provider determines the price $x_v$ for each participant group $v \in V$. 
Then, the capacity $\xi_v$ of participant group $v \in V$ is drawn from a given distribution, which is $\mathrm{Bin}(n_v, p_{v}(x_v))$ or $\Po(n_v p_v(x_v))$.\footnote{In our previous study \citep{hikima2021integrated}, $v$ represents each participant (rather than each participant group). In this problem setting, prices cannot be determined until node $v$, i.e., each participant, has been revealed. 
This causes an overhead until a price is quoted to each participant: if a participant requests a taxi ride on the ride-sharing platform, the price cannot be known until information on other participants is gathered. 
In contrast, in the problem setting of this paper, instead of setting a price for each participant, we set a price for each group of participants (before a participant appears). 
This allows us to quote a price as soon as a participant appears, which reduces the overhead time. } 
Here, $\mathrm{Bin}(n_v, p_{v}(x_v))$ is a binomial distribution with parameters $n_v$ and $p_{v}(x_v)$, and $\Po(n_v p_v(x_v))$ is a Poisson distribution with parameter $n_v p_{v}(x_v)$, where $n_v$ is a non-negative integer and $p_v:\mathcal X_v \to \mathcal Z_v \subseteq [0,1]$ is a real-valued function. 
Here, $\mathcal X_v$ is an interval for each $v \in V$. 
After the realization of $\xi_v\ (v \in V)$, the platform provider decides a b-matching\footnote{A \textit{b-matching} of a given graph $G$ is an assignment of integer weights to the edges of $G$ so that the sum of the weights on the edges incident with a node $v$ is at most $b_v$. In our problem setting, $b_v=\xi_v$ for $v \in V$ and $b_u=c_u$ for $u \in U$.} $\Pi$ to maximize the total reward while satisfying resource capacity conditions (the number of selected edges connected to node $u$ is less than or equal to $c_u$) and participant capacity conditions (the number of selected edges connected to node $v$ is less than or equal to 
 $\xi_v$). Accordingly, the total reward is $\sum_{e=(u, v) \in \Pi} (x_v+w_e)$. 

Here, let us use several real-world problems to elucidate the above procedure. 

\paragraph{Ride-hailing platform.} In a ride-hailing platform, the provider needs to decide one-to-one matchings between multiple requesters and multiple taxis in real-time. Here, we can divide the time horizon into multiple time steps and consider that taxi dispatches are to be determined at each time step. There are multiple taxis $U$ with unit capacity ($c_u=1$ for all $u\in U$) and requester groups $V$, where each group is defined by the departure area and the trip distance. Let $w_{(u,v)} (\le 0)$ be the total cost of allocating taxi $u \in U$ to requester $v \in V$, including the cost of gasoline, opportunity costs, and other cost factors.\footnote{ The study by \cite{Tong2018DynamicPI} deals with a case in which $w_{(u,v)}$ is $0$ or $-\infty$ for all $u$ and $v$. Therefore, our problem setting is a generalization of that of \citep{Tong2018DynamicPI}. 
} The provider can determine the price $x_v (\ge 0)$ for each requester group $v \in V$. When $n_v$ requesters appear, they accept price $x_v$ with probability $p_v(x_v)$ for each $v \in V$; that is, the demand $\xi_v \in \{0,1,\dots,n_v\}$ of each requester group $v \in V$ follows a binomial distribution Bin($n_v,p_v(x_v)$).\footnote{
Note that a binomial distribution (including a Bernoulli distribution) was used to represent the demand in \citep{feng2023two, shah2022joint, Tong2018DynamicPI}, whereas a Poisson distribution could also be used.} The provider allocates taxi $u$ to requester $v$ while satisfying both the taxis' and requesters' capacity conditions and gets a profit of $w_{(u,v)}+x_v$.

\paragraph{Crowd-sourcing platform.} \noindent In a crowd-sourcing platform, the provider needs to decide task-to-worker matchings. 
There are task types $U$ with capacity $c_u \in Z_{>0}$ for $u \in U$ and multiple worker types $V$. Let $w_{(u,v)} (\ge 0)$ be the reward paid by the requester of the tasks to the provider when a task $u$ is completed by worker $v$. 
A reward $w_{(u,v)}$ is calculated on the basis of the skills and performance of each worker. The provider can determine a wage $x_v (\le 0)$ for each worker type $v \in V$. Here, $x_v$ is negative because it is the price the provider pays for worker $v$. Then, we assume the capacity $\xi_v \in \{0,1,\dots,\infty \}$ of each worker type $v$ follows a Poisson distribution with intensity $n_v p_v(x_v)$.\footnote{In some studies \citep{huang2021bike,raviv2018crowd}, worker arrivals are assumed to follow a Poisson distribution.} The provider allocates tasks to workers of type $v$ while satisfying the tasks’ and workers’ capacity conditions and gets a profit of $w_{(u,v)}+x_v$.

\subsection{Optimization problem} 
\label{sec:optimization_problem}
For the above procedure, let us consider an optimization problem to maximize the profit of the platform provider.
\begin{align*}
{\rm (P)} \quad \max_{\bx \in \mathbb{R}^{V}}\ \ & \mathbb{E}_{\bmxi \sim D(\bx)}[f(\bx,\bmxi)], \\
{\rm s.t. } \ \ &x_v \in {\mathcal X}_v , \quad \forall v \in V,
\end{align*}
where $f(\bx,\bmxi)$ is the optimal value of the following problem:
\begin{align}
{\rm (P_{sub})}\ \ \ \ \displaystyle{\max_{\bz\in \Z_{\ge 0}^{E}}} \ \ & \sum\nolimits_{e=(u,v) \in E} (w_e+x_{v}) z_e \nonumber \\
{\rm s.t. } \ \ &\displaystyle{ \sum\nolimits_{e \in \delta(v)} } z_e \le \xi_v, \ \ \ \forall v \in V, \nonumber\\
 &\sum\nolimits_{e \in \delta(u)} z_e \le c_u, \ \ \ \forall u \in U, \nonumber
\end{align}
which is a bipartite b-matching problem for given $\bx$ and $\bmxi$. Here, $\delta(a)$ is the non-empty set of edges connected to node $a$.\footnote{We can assume $\delta(a) \neq \emptyset$ for all $a \in U \cup V$ without loss of generality, since removing $\{ a \mid \delta(a) = \emptyset \}$ from $U$ or $V$ does not affect the optimization problem.} The set ${\mathcal X}_v$ is an interval and the domain of $p_v(\bx)$. The distribution $D(\bx)$ is defined by $D(\bx):=\prod_{v \in V} D_{v}(\bx)$, where $D_{v}(\bx)$ is a given distribution. We assume that $D_{v}(\bx)$ is a binomial distribution Bin($n_v,p_v(x_v))$ for each $v \in V$ or a Poisson distribution Po($n_vp_v(x_v))$ for each $v \in V$. Moreover, without loss of generality, we can assume that there exists $x_v \in \mathcal X_v$ such that $x_v > -\max_{e \in \delta(v)} w_e$ for each $v \in V$.\footnote{If such $x_v$ does not exist for some $v \in V$, then $w_{(u,v)} + x_v$ is not positive for any $u \in U$ and $x_v \in \mathcal X_v$. When $w_{(u,v)} + x_v$ is not positive, $z_v^*=0$ for any optimal solution $\bz^*$ of ${\rm (P_\mathrm{sub})}$. Such $v$ can be removed from our problem without affecting the solution.}

\section{Preliminaries}
\subsection{Assumption and Preliminary Lemma}
Throughout the paper, we let 
\begin{align}
\bar{\xi}_v(x):=n_vp_v(x), \label{def:bar_xi}
\end{align}
which is the expected value of $\xi_v$.

Moreover, we make the assumption below. 

\begin{assumption} \label{asp:average_rand}
The following holds for all $v \in V$.
\begin{enumerate}
\item $p_v$ is a differentiable, strictly decreasing, and bijective function.
\item When $\mathcal X_v$ is bounded from above, $p_v(\max {\mathcal X}_v) = 0$. When $\mathcal X_v$ is unbounded from above, $\lim_{x \rightarrow \infty} p_v(x) = 0$.
\item $p_v'(x)/p_v(x)$ is monotonically non-increasing for all $x \in {\Int}(\mathcal X_v)$, where $p_v'(x)$ is the derivative of $p_v(x)$.
\end{enumerate} 
\end{assumption}

Assumption~\ref{asp:average_rand} is not so restrictive; for instance, many studies \citep{carvalho2005learning,dong2009dynamic,schulte2020price} have assumed that the acceptance probability for price has properties (a)--(c). Specifically, functions satisfying property (c) are called \textit{log-concave functions} \citep{bagnoli2006log} or \textit{monotone hazard rate functions} \citep{barlow1963properties}, and they include the cumulative distribution functions of many distributions. For example, $p_v(x) = C F(x)$ satisfies Assumption~\ref{asp:average_rand}, where $C$ is a constant and $F(x)$ is a Gauss error function or a logistic function.

Under Assumption \ref{asp:average_rand}, we can prove the following lemma, which will be useful for our algorithm.

\begin{lemma} \label{lem:inverse_concave}
Suppose that Assumption \ref{asp:average_rand} holds. Then, $\bar{\xi}_v'(x) < 0$ for any $x \in \Int(\mathcal X_v)$ and $v \in V$. Moreover, $-\bar{\xi}_v^{-1}(z) z$ is convex w.r.t. $z$ in the domain $\mathcal Z_v$ for all $v \in V$.
\end{lemma}

\subsection{Min-cost flow problem 
\label{mcrw}}
Here, we briefly describe the min-cost flow problem, since our algorithm derives a $(1-1/e)$-approximation solution from an optimal solution of one. Let $\hat{G}=(\hat{V},\hat{E})$ be a directed graph with a cost function $c_{(i,j)}: \mathbb{R} \rightarrow \mathbb{R}$ and capacity $\ell_{(i,j)} \in \mathbb{R}_{\geq 0}$ associated with each edge $(i,j) \in \hat{E}$. Each node $i \in \hat{V}$ has a value, $b_i \in \mathbb{R}$, which is called the supply of the node when $b_i>0$, or the demand of the node when $b_i<0$. Given the above, the min-cost flow problem can be written as follows:
\begin{align*}
{\rm (MCF)}\ \min_{\bz \in \mR^{\hat{E}}}\quad & \sum_{(i,j) \in \hat{E}} c_{(i,j)}(z_{(i,j)}) \\
{\rm s.t.} \quad \ 
& \sum_{j : (i,j) \in \hat{E}} z_{(i,j)} - \sum_{j : (j,i) \in \hat{E}} z_{(j,i)} = b_i, \ \ \forall i \in \hat{V}, \\
& 0 \le z_{(i,j)} \le \ell_{(i,j)}, \ \ \forall (i,j) \in \hat{E}. 
\end{align*}

Min-cost flow problems are generally NP-hard and difficult to solve when no assumptions are placed on the cost function $c_{(i,j)}$. However, several methods can be used to solve them efficiently when the cost functions are convex \citep{kiraly2012efficient,ahuja1988network}. In Section \ref{capacity}, we introduce one of those methods, i.e., the capacity scaling algorithm, to find an optimal solution of the min-cost flow problem derived from (P).

\section{Proposed algorithm \label{Prom}}
Sections \ref{Aof}--\ref{capacity} describe the approximation algorithm for (P). In particular, Section \ref{Aof} approximates the objective function of (P) and devises an optimization problem (PA) whose optimal solution is a $(1-1/e)$-approximation solution for (P). Section \ref{subsec:FP} shows that (PA) can be reduced to a convex min-cost flow problem (FP) under Assumption~\ref{asp:average_rand}, and Section \ref{capacity} describes the existing methods for solving convex min-cost flow problems.

\subsection{Approximation of objective function} 
\label{Aof}
Here, we seek an approximation of the objective function of (P), that is, $\mathbb{E}_{\bmxi \sim D(\bx)}[f(\bx,\bmxi)]$. First, we introduce the following problem for a given $\bx$:
\begin{align}
{\displaystyle \max_{\bz \in \R_{\ge 0}^E}}&\ \ \ \sum\nolimits_{e=(u,v) \in E} (x_v+w_e) z_e \label{low2} \\
\textrm{s.t.}&\ \ \sum\nolimits_{e \in \delta(v)} z_e \le \bar{\xi}_v(x_v), \ \ \ \forall v \in V, \nonumber\\
 &\ \ \sum\nolimits_{e \in \delta(u)} z_e \le c_u, \ \ \ \forall u \in U, \nonumber
\end{align}
where $\bar{\xi}_v(x_v)$ is defined by \eqref{def:bar_xi}.

Let $\hat{f}(\bx)$ be the optimal value of the above problem. Then, the following theorem holds.\footnote{The current paper is based our previous paper \citep{hikima2021integrated}, in which we prove that the lower bound is $\frac{1}{3} \hat{f}(\bx)$, but the results of a subsequent study \citep{brubach2021improved} updated it to $(1-1/e) \hat{f}(\bx)$.}
\begin{theorem} \label{thm:approximation_ratio}
The following holds for any $\bx$:
\[(1-1/e) \hat{f}(\bx) \le \mathbb{E}_{\bmxi \sim D(\bx)}[f(\bx,\bmxi)] \le \hat{f}(\bx)\vspace{1.5mm}\]
\end{theorem}

We obtain the optimization problem below by replacing $\mathbb{E}_{\bmxi \sim D(\bx)}[f(\bx,\bmxi)]$ with $\hat{f}(\bx)$ for (P):
\begin{align*}
{\rm (PA)}\ \ \ {\displaystyle \max_{\bx \in \R^V,\ \bz\in \R_{\ge 0}^E}}&\ \ \ \sum\nolimits_{e=(u,v) \in E} (x_v+w_e) z_e \\
\textrm{s.t.} \quad & \quad x_v \in \mathcal X_v, \ \ \ \forall v \in V, \\
&\ \ \sum\nolimits_{e \in \delta(v)} z_e \le \bar{\xi}_v(x_v), \ \ \ \forall v \in V, \nonumber\\
 &\ \ \sum\nolimits_{e \in \delta(u)} z_e \le c_u, \ \ \ \forall u \in U. \nonumber
\end{align*}
Here, the following lemma holds.
\begin{lemma} \label{lem:opt_exist}
Under Assumption~1, (PA) has an optimal solution. 
\end{lemma}
From Theorem~\ref{thm:approximation_ratio}, the solution of (PA) is a $(1-1/e)$-approximation solution for (P).

\subsection{Reduce (PA) to a convex min-cost flow problem}
\label{subsec:FP}
Although (PA) is a non-convex problem with non-convex functions $\bar{\xi}_v(x_v)$, we can reduce (PA) to a convex min-cost flow problem under \mbox{Assumption~\ref{asp:average_rand}}. First, let us consider 
\begin{align*}
{\rm (PA')}\ \ &{\displaystyle \max_{\bz\in \R_{\ge 0}^E}}\ \ \sum_{v \in V} \bar{\xi}_v^{-1} \left( \sum_{e \in \delta(v)} z_e \right) \sum_{e \in \delta(v)} z_e + \sum_{e \in E} w_e z_e \\
&{\rm s.t.}
\quad \sum_{e \in \delta(v)} z_e \in \mathcal Z_v, \ \ \forall v \in V, \\ 
 &\quad \quad \sum_{e \in \delta(u)} z_e \le c_u, \ \ \forall u \in U.
\end{align*}
This optimization problem eliminates the decision variable $\bx$ in (PA) by making the substitution $x_v:=\bar{\xi}_v^{-1}(\sum\nolimits_{e \in \delta(v)} z_e )$.

We can prove the following theorem regarding the solution of this problem.

\begin{theorem} \label{thm:pa_p_solution}
Suppose that Assumption~\ref{asp:average_rand} holds. Let the optimal solution of ${\rm (PA')}$ be $\bz^*$ and $x^*_v:=\bar{\xi}_v^{-1}(\sum\nolimits_{e \in \delta (v)} z_e^* )$ for all $v \in V$. Then, $(\bx^*,\bz^*)$ is an optimal solution for ${\rm (PA)}$.
\end{theorem}
From Theorem~\ref{thm:pa_p_solution}, we can solve (PA) by solving ${\rm (PA')}$.

Next, we reduce ${\rm (PA')}$ to a min-cost flow problem. We prepare slack variables $z_{(s,u)}$ and $z_{(v,t)}$ with new subscripts $s$ and $t$. Here, we let $z_{(s,u)}:=\sum_{e \in \delta(u)} z_e$ for all $u \in U$ and $z_{(v,t)}:=\sum_{e \in \delta(v)} z_e$ for all $v \in V$. In addition, we let $z_{(s,t)}$ be a slack variable and $C:=\min \{\sum_{v \in V} \sup \mathcal Z_v,\ \sum_{u \in U} c_u \}$. Accordingly, ${\rm (PA')}$ can be written as 
\begin{align}
{\rm (FP)}\ \ \ \min_{\substack{
\bz \in \mR^E,\ \bz^s \in \mR^U, \\
\bz^t \in \mR^V,\ z_{(s,t)}\in \mR}}\ \ &\sum_{v \in V} -\bar{\xi}_v^{-1} (z_{(v,t)}) z_{(v,t)} - \sum_{e \in E} w_e z_e \label{FP1} \\
{\rm s.t. } \ \ 
& \sum_{u\in U} z_{(s,u)} + z_{(s,t)} = C,\ \sum_{v \in V} z_{(v,t)} +z_{(s,t)} = C, \label{FP2}\\
& z_{(s,u)}-\sum_{e \in \delta(u)} z_e = 0, \ \ \ \forall u \in U, \label{FP3}\\
& \sum_{e \in \delta(v)} z_e - z_{(v,t)} = 0, \ \ \ \forall v \in V, \label{FP4}\\
& \sum_{e \in \delta(v)} z_e \in \mathcal Z_v, \ \ \forall v \in V,\label{FP5-1}\\ 
& 0 \le z_e , \ \ \ \forall e \in E, \label{FP7} \\
&0 \le z_{(s,u)} \le c_u, \ \ \forall u \in U, \label{FP7-2} \\
& 0 \le z_{(s,t)} \le C, \label{FP8}
\end{align}
where $\bz^s\coloneqq \{z_{(s,u)}\}_{u \in U}$ and $\bz^t \coloneqq \{z_{(v,t)}\}_{v \in V}$. This is a nonlinear min-cost flow problem for a graph with $U \cup V \cup \{ s, t\} $ as nodes.

Now let us explain (FP) in the context of the min-cost flow problem (See Fig. \ref{flow}). First, \eqref{FP1} represents the cost function for the flow amount of each edge. $ -\bar{\xi}_v^{-1} (z_{(v,t)}) z_{(v,t)}$ is the cost function for edge $(v,t)$ and $-w_e z_e$ is the cost function for edge $e \in E$. The cost for $z_{(s,u)}$ and $z_{(s,t)}$ is $0$ because these variables are not in the objective function. Second, \eqref{FP2}, \eqref{FP3}, and \eqref{FP4} represent the demand/supply for each node. From these equations, the supply at node $s$ is $C$, the demand at node $t$ is $-C$, and the demands/supplies at other nodes are $0$. Therefore, the problem is to find a way of sending a flow amount equal to $C$ from node $s$ to node $t$ through the network. Third, \eqref{FP5-1}, \eqref{FP7}, \eqref{FP7-2}, and \eqref{FP8} represent the flow capacity of each edge. As shown in \mbox{Fig. \ref{flow}}, we can view (FP) as a min-cost flow problem.

\begin{figure}[t]
\centering
 \includegraphics[scale=0.5]{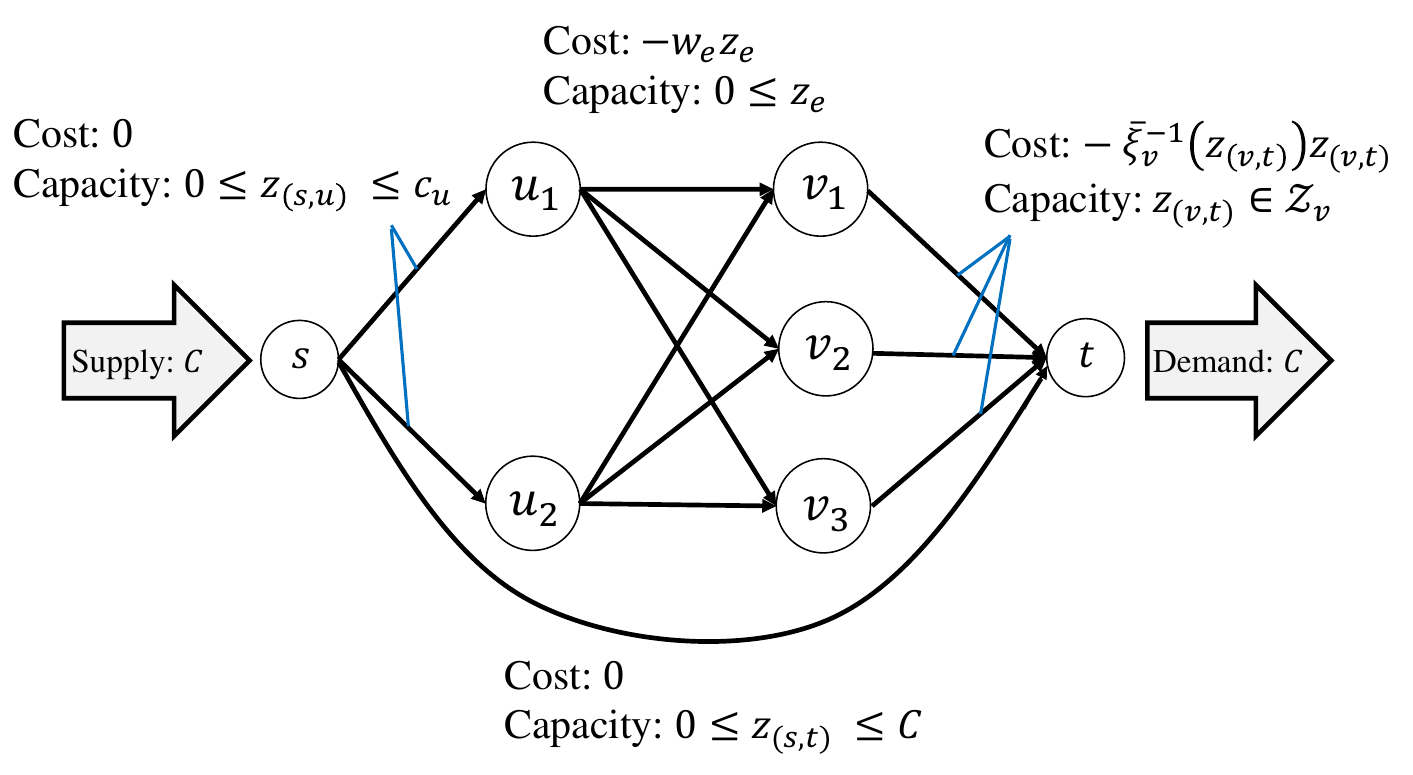}
\caption{Illustration of (FP)}
\label{flow}
\end{figure}

The theorem below follows immediately from Lemma \ref{lem:inverse_concave}.
\begin{theorem} \label{thm:fp_convex}
When Assumption~\ref{asp:average_rand} holds, the cost function for all edges in the min-cost flow problem (FP) is convex. Consequently, (FP) is a convex min-cost flow problem.
\end{theorem}

\subsection{Solution for (FP) via capacity scaling algorithm} 
\label{capacity}
From Theorem~\ref{thm:fp_convex}, we can use the capacity scaling algorithm \citep{ahuja1988network,vegh2016strongly} to solve (FP), as described in Section \ref{mcrw}. The capacity scaling algorithm, as explained in \citep{ahuja1988network}, solves convex min-cost flow problems that restrict the feasible values of $\bz$ to integers. We can use this method to obtain an optimal solution in a continuous domain to any desired degree of accuracy by using the following procedure: (i) Substitute $\Delta y_e$ for each $z_e$ for each edge $e$, where $\Delta \in \R_{\ge 0}$ is an allowable error and $y_e \in \mathbb{Z}$; (ii) find an integer optimal solution $\by^*$ of the transformed problem; (iii) let $z_e^{\dagger}:=\Delta y_e^*$. Then, $\bz^{\dagger}$ is an optimal solution of the original problem with a degree of accuracy \mbox{of $\Delta$}, that is, $|z_e^{\dagger}-z_e^*| \le \Delta$ for all $e \in E$, where $\bz^*$ is some optimal solution for (FP). Moreover, the capacity scaling algorithm in \citep{vegh2016strongly} can find an optimal solution in strongly polynomial time when the cost functions are quadratic. We can use the method of \citep{vegh2016strongly} on our problem when $p_v$ is a linear function (such as $p_v^{L}$ in Section \ref{subsec:para_setup}).

Now we come to our proposal, Algorithm \ref{alg:FP_to_solution}. From Theorem~\ref{thm:approximation_ratio}, Theorem~\ref{thm:pa_p_solution}, and the fact that (PA${\rm'}$) is equivalent to (FP), it can be shown that Algorithm \ref{alg:FP_to_solution} outputs a $(1-1/e)$-approximation solution for (P).

\begin{algorithm}[t]
 \caption{Proposed algorithm}
 \label{alg:FP_to_solution}
 \begin{algorithmic}[1]
 \Require{Constants $w_{(u,v)}, c_u$, and function $\bar{\xi}_v$ for all $u \in U$ and $v \in V$}
 \State{Solve (FP) and let $\bz^*$ be an optimal solution} \label{state:FP_solve}
 \State{Let $x^*_v:=\bar{\xi}_v^{-1}(\sum\nolimits_{e \in \delta (v)} z^*_e)$ for all $v \in V$} \label{state:z_to_x_convert}
 \State{\textbf{Return} $\bx^*$}
 \end{algorithmic}
 \end{algorithm}

\section{Experiments} \label{exp}
We conducted experiments to show that:
{
\setlength{\leftmargini}{10pt}
\begin{itemize}
	\setlength{\itemsep}{1pt}
	\setlength{\parskip}{0pt}
	\setlength{\itemindent}{0pt}
	\setlength{\labelsep}{3pt}
\item The proposed algorithm outputs a more profitable solution compared with the other methods in each application.
\item The proposed algorithm outputs a solution in a practical amount of time.
\end{itemize}}

\noindent We performed simulation experiments using real data from a ride-hailing platform and a crowd-sourcing platform. The experiments were run on a computer equipped with a Xeon Platinum 8168 of 4 x 2.7GHz, 1TB of memory, and CentOS 7.6. The program codes were implemented in Python.

\subsection{Ride-hailing platform \label{Rpp}}
We conducted experiments on a ride-haling platform whose matching procedure is described in Section \ref{sec:matching_procedure}.

\subsubsection{Data sets and parameter setup} 
\label{subsec:para_setup}
We used ride data gathered in New York City\footnote{https://www1.nyc.gov/site/tlc/about/tlc-trip-record-data.page}. We used yellow taxi data and green taxi data from Manhattan, Queens, Bronx, and Brooklyn. Each record consists of pick-up area, pick-up time, drop-off area, drop-off time, trip distance, and total amount charged to passengers. We performed simulations using data from a holiday and weekday in a randomly chosen week: October 6 and 10, 2019. For each day, we constructed requester-taxi matching situations  every 5 minutes from 10:00 to 20:00, that is, we simulated $120$ ($10$ hours $\times$ $12$ times) situations. We chose a time step, $t_s$, which is required to generate the situation, of 30 seconds for Manhattan and 300 seconds for regions other than Manhattan because of the differences in the amount of data in each region. The features of the dataset that we used are summarized in Table \ref{dataset}.

\begin{table}[t]
\begin{center}
\small
 \begin{tabular}{ccrrrr} \hline
 Region \& & date & \multicolumn{2}{c}{$|U|$ (taxis)} & \multicolumn{2}{c}{$|V|$ (requesters) } \\ \cmidrule(lr){3-4} \cmidrule(lr){5-6}
 Time interval & & MEAN & \multicolumn{1}{c}{SD} & MEAN & \multicolumn{1}{c}{SD} \\ \hline 
Manhattan, & 10/06&76.1&13.1&80.0&11.1 \\ 
$30$ seconds & 10/10&96.6&21.6&100.1&17.4 \\ \hline
Queens, &10/06&88.3&27.0&92.3&22.7 \\ 
$300$ seconds &10/10&89.8&28.1&95.4&23.1 \\ \hline
Bronx, &10/06 &5.3&2.7&5.3&2.8\\ 
$300$ seconds &10/10 &6.2&2.7&6.2&2.6\\ \hline
Brooklyn, & 10/06 &26.4&5.6&26.5&5.7\\ 
$300$ seconds & 10/10&26.7&7.3&27.2&5.3\\ \hline
\end{tabular}
\end{center}
\caption{Summary of real dataset of ride-hailing platform.}\label{dataset}
\end{table}

We recreated situations from the data of each region and set the inputs $U$, $V$, $w_{(u,v)}$, and the distribution of $\xi_v$ at each time.\\
(i) $U$: We assumed that it is possible to dispatch taxis that have completed a request, and extracted taxi ride data from the dataset that had a drop-off time within $t_s$ seconds from the target minute as $U$. The location of each taxi $u \in U$ was set by adding Gaussian noise to the center point of the drop-off area. This is because the taxi ride data records only the areas of the pick-up/drop-off. \\
(ii) $V$: We extracted taxi ride data that had a pick-up time within $t_s$ seconds from the target minute as the requester set $V$. We generated the origin/destination points for each requester $v \in V$ by adding Gaussian noise to the center point of the pick-up/drop-off area.
\\
(iii) $w_{(u,v)}$: Let $w_{(u,v)}=-18.0 \tau_{(u,v)}$. 
Here, $\tau_{(u,v)}$ is the time required for taxi $u$ to fulfill request $v$, and it is calculated from the destination and origin of requester $v$ and the location of taxi $u$. 
The parameter value $18.0$ was the taxi driver's opportunity cost that was based on the driver's income.\\
(iv) Distribution of $\xi_v$: We defined $p_v$ in two ways. 
The first way was as a linear function $p_v^{L}:[q_v, 1.5 q_v] \to [0, 1]$ such that 
\[p_v^{L}(x):= -\frac{2}{q_v} x + 3,
\]
where $q_v$ is a constant scalar. Here, $q_v$ is the actually paid amount for each request $v$ in the data set, and we assumed that $\xi_v(x_v) \sim \textrm{Bin}(1, p^{L}_v(x_v))$. The second way defined $p_v^{S}(x):(-\infty,\infty) \to (0,1)$ with the following sigmoid model:
\[p_v^{S}(x):= 1- \frac{1}{1+e^{- (x-\beta q_v)/(\gamma |q_v|) }},\]
where $\beta$ and $\gamma$ are constants, and we set $\beta=1.3$ and $\gamma=0.3 \sqrt{3}/\pi$. The same $q_v$ as in the linear function was used. Moreover, we assumed that $\xi_v(x_v) \sim \textrm{Bin}(1, p^{S}_v(x_v))$.

\subsubsection{Metric}
We used the approximated objective value (that is, approximated expected benefits of the platform provider): ${\rm obj}:=\frac{1}{10^2} \sum_{\ell=1}^{10^2} f(\hat{\bx},\bmxi^\ell(\hat{\bx}))$, where $\hat{\bx}$ is the output of each method and $\bmxi^\ell(\hat{\bx}) \sim D(\hat{\bx})$, as a metric to calculate the expected profit obtained at each time.

\subsubsection{Baselines} 
\label{subsec:rideshare_experiment}
We compared the proposed algorithm with two state-of-the-art methods.

\noindent{\bf MAPS \citep{Tong2018DynamicPI}:} It is an approximation algorithm for area-basis pricing of taxi services. 
We divided up the requesters into groups by area and defined the acceptance probability ($S^g(p)$ in \citep{Tong2018DynamicPI}) for each area by taking the average of the individual acceptance probability functions within the area.
Moroever, we assumed that each taxi is dispatchable for a requester whose distance is within $2$ km from the taxi.
Then, we used MAPS to obtain the prices. 

\noindent{\bf LinUCB \citep{li2010contextual}:} It is a generic contextual bandit algorithm used by \citep{chen}. As the arms of the method, we chose pricing factors that were multiples of the base price, $\{0.6, 0.8, 1.0 , 1.2, 1.4\}$. The base price for each request was calculated according to the trip distance and the average price per unit distance for all requests in the period of July to September, 2019. As features for learning, we used pick-up areas, drop-off areas, hours, and trip distance. The initial parameter $\theta_\tau$ for each arm $\tau \in \{0.6, 0.8, 1.0, 1.2, 1.4\}$ was learned through $11040$ runs using request data from July to September, 2019.

\begin{table}[t!]
\begin{center}
\begin{tabular}{crrrrrrr} \hline
 Place \& &\multicolumn{1}{c}{Day} 
 & \multicolumn{2}{c}{Proposed} 
&\multicolumn{2}{c}{MAPS}
&\multicolumn{2}{c}{LinUCB} 
\\ \cmidrule(lr){3-4} \cmidrule(lr){5-6} \cmidrule(lr){7-8} 
$p_v$ & &\multicolumn{1}{c}{obj} &\multicolumn{1}{c}{time(s)}& \multicolumn{1}{c}{obj}&\multicolumn{1}{c}{time(s)}&\multicolumn{1}{c}{obj}&\multicolumn{1}{c}{time(s)} \\ \hline
Manhattan,&10/6&{\bf 1.07e+3}&5.18e+0&7.69e+2&9.84e-2&5.66e+2&9.27e+0 \\ 
Linear&10/10&{\bf 1.46e+3}&8.70e+0&1.01e+3&1.21e-1&7.56e+2&1.16e+1 \\ \hline
Manhattan,&10/6&{\bf 9.82e+2}&6.62e+0&7.33e+2&1.25e-1&6.00e+2&9.41e+0 \\ 
Sigmoid&10/10&{\bf 1.30e+3}&1.11e+1&9.67e+2&1.47e-1&7.99e+2&1.16e+1 \\ \hline
Queens,&10/6&{\bf 2.18e+3}&6.07e+0&3.84e+2&5.55e-2&7.84e+2&1.18e+1 \\ 
Linear &10/10&{\bf 2.44e+3}&6.90e+0&4.41e+2&6.02e-2&1.06e+3&1.20e+1 \\ \hline
Queens, &10/6&{\bf 1.95e+3}&5.47e+0&3.80e+2&8.13e-2&1.08e+3&1.14e+1 \\ 
Sigmoid&10/10&{\bf 2.23e+3}&7.79e+0&4.49e+2&9.09e-2&1.31e+3&1.15e+1 \\ \hline
Bronx, &10/6&{\bf 6.83e+1}&8.06e-3&2.35e+1&3.89e-3&2.66e+1&1.84e-1 \\ 
Linear &10/10&{\bf 9.42e+1}&1.03e-2&3.70e+1&5.46e-3&3.89e+1&2.19e-1 \\ \hline
Bronx,&10/6&{\bf 6.37e+1}&1.05e-2&2.46e+1&5.02e-3&3.67e+1&1.89e-1 \\ 
Sigmoid&10/10&{\bf 8.73e+1}&1.34e-2&3.83e+1&6.85e-3&5.00e+1&2.32e-1 \\ \hline
Brooklyn,&10/6&{\bf 3.37e+2}&2.37e-1&1.83e+2&2.24e-2&1.42e+2&2.48e+0 \\ 
Linear &10/10&{\bf 3.98e+2}&2.66e-1&2.06e+2&2.84e-2&1.69e+2&2.44e+0 \\ \hline
Brooklyn,&10/6&{\bf 3.04e+2}&2.73e-1&1.76e+2&2.91e-2&1.62e+2&2.46e+0 \\ 
Sigmoid &10/10&{\bf 3.58e+2}&3.56e-1&1.96e+2&3.71e-2&1.91e+2&2.52e+0 \\ \hline
 \end{tabular}
\end{center}
\caption{Results of ride-hailing platform simulations. Each result represents the average of $120$ dispatch runs. The best value for each dataset in obj is in bold. } \label{real}
\end{table}

\subsubsection{Experimental results}
Table \ref{real} shows the results of the experiments. Regardless of the region or form of the function $p_v$, the proposed algorithm outperformed all of the compared methods in terms of the objective value. As for the computational time, the proposed algorithm solved the problem quickly enough for practical use, although MAPS took less time.

\subsection{Crowd-sourcing market}
We conducted experiments in a crowd-sourcing platform whose matching procedure is described in Section~\ref{sec:matching_procedure}.

\subsubsection{Data set and parameter setup}
We used an open crowd-sourcing dataset \citep{Buckley10-notebook}. The data set contains records of worker's judgments on the task of checking the relevance of a given topic and a web page. Each record has elements of (topic ID, worker ID, document ID, judgment). The judgments are divided up into five categories: highly relevant, relevant, non-relevant, unknown, and broken link. Here, broken link indicates that the web page cannot be viewed. This data set is summarized in Table \ref{dataset2}. We used this data to replicate the worker-task matching situations and conducted multiple experiments with it. 

\begin{table}[t]
\begin{center}
 \begin{tabular}{rrrrr} \hline
 All data&Topic ID & Document ID & \multicolumn{1}{c}{Worker ID} & \multicolumn{1}{c}{Judgments} \\ \hline
98453 &100 &19902 &766 &5 \\ \hline
\end{tabular}
\end{center}
\caption{Summary of the real dataset in crowd-sourcing market.}\label{dataset2}
\end{table}

In the experiments, we varied the parameters $(\phi, \psi)$ and $n$ and set conditions on the inputs, $U$, $V$, $w_{(u,v)}$, and the distribution of $\xi_v$ from the data as follows.

\noindent (i) $U$: Each task in the data was assumed to appear with a probability of $\psi$ with $U$ being the set of tasks created.

\noindent (ii) $V$: Each worker in the data was assumed to be active with a probability of $\phi$ with $V$ being the set of active workers.

\noindent (iii) $w_{(u,v)}$: We decided the correct judgment for each task in the data by majority vote. Let $\phi^v_s$ be the percentage of correct answers of worker $v$ for topic $s$. In our experiment, we assumed that $\phi^v_s$ are known a priori and $w_{(u,v)}:=\phi^v_{s(u)}$ for each $(v,\ u)$. Here, $s(u)$ means the topic of task $u$. This setup is based on a scheme that determines the value of a task depending on the skills of the worker. For topics that worker $v$ had never worked on, we set the percentage of correct answers of worker $v$ to be that for all of the tasks that worker $v$ performed.

\noindent (iv) Distribution of $\xi_v$: We defined $p_v$ in two ways. The first way was as a linear function $p_v^{L}:[1.5 q_v, q_v] \to [0, 1]$ such that
\begin{align*}
p_v^{L}(x):= \frac{2}{q_v} x - 2. \label{eq:p_v_PL}
\end{align*}
Since the data does not contain information on the amount paid to each worker $v$, we chose $q_v$ from a uniform distribution in $[-0.4,-0.1]$ and assumed that $\xi_v(x_v) \sim \textrm{Bin}(1, p^{L}_v(x_v))$. The second way used the following sigmoid model:
\[p_v^{S}(x):= 1- \frac{1}{1+e^{- (x-1.25 q_v)\pi/(0.25 |q_v|) }},\]
where the same setting of $q_v$ as in the linear function was used. We assumed that $\xi_v(x_v) \sim \textrm{Bin}(1, p^{S}_v(x_v))$.

\subsubsection{Metric}
We ran $10^3$ simulations for each setting. Then, we took the average of the resulting profits of $10^3$ simulations be the metric, obj, and regarded it to be an approximation of the expected benefits for the platform provider.

\subsubsection{Baselines}
We compared the proposed algorithm with two existing methods.

\noindent{\bf Myerson Reserve Price (MRP) \citep{myerson1981optimal}:} MRP is the optimal price in a single-item market when there is enough supply and there is no differentiation among buyers. In using this price, we considered multiple types of task to be one type and all workers to be of average ability. Specifically, we set $x_v:= {\rm arg max}_x \{ (x+ \hat{w} ) p(x) \}$ for all $v$, where $p(x):=\sum_{v \in V} p_v(x)$, and $\hat{w}$ is the average of the correct rate for all combinations of workers and tasks.

\noindent{\bf Capped MRP:} It is an analogue of the existing pricing strategy \citep{babai2015} created to tackle the problem of limited supply in single item-market when there is no differentiation among buyers. 
We considered multiple types of task to be one type and all workers to be of average ability. 
The method \citep{babai2015} determines the price while estimating $p(x)$, which is the acceptance probability of each worker for payment $x$; here, we took it to be a given function. 
Specifically, for all $v$, we set $x_v:={\rm arg max}_x \{ (x+\hat{w}) \min (|U|, |V| p(x)) \}$, where $p(x)$ and $\hat{w}$ are the same as those defined for MRP.

\subsubsection{Experimental results}
Table \ref{real2} shows the results of the simulation experiments with different parameter values. Regardless of the problem parameters or form of the function $p_v$, the proposed algorithm outperformed all baselines in terms of the objective value. In addition, the computational time of the proposed algorithm was short enough for practical use, although the compared methods required less computational time.

\begin{table}[t!]
\begin{center}
\begin{tabular}{ccrrrrrr} \hline
$(\phi, \psi)$&$p_v$& \multicolumn{2}{c}{Proposed} &\multicolumn{2}{c}{MRP} &\multicolumn{2}{c}{Capped  MRP} 
\\ \cmidrule(lr){3-4} \cmidrule(lr){5-6} \cmidrule(lr){7-8}
&&\multicolumn{1}{c}{obj} &\multicolumn{1}{c}{time(s)}& \multicolumn{1}{c}{obj}&\multicolumn{1}{c}{time(s)}&\multicolumn{1}{c}{obj}&\multicolumn{1}{c}{time(s)} \\ \hline
 (0.10, 5.0e-4)&Linear&{\bf 1.88e+1}&6.12e+0&1.34e+1&8.63e-3&1.32e+1&9.03e-3\\ 
&Sigmoid&{\bf 1.84e+1}&2.25e+1&1.26e+1&1.87e-2&1.40e+1&1.91e-2\\ \hline
 (0.10, 1.0e-3)&Linear&{\bf 2.09e+1}&1.91e+1&1.36e+1&1.09e-2&1.36e+1&1.13e-2\\ 
&Sigmoid&{\bf 2.03e+1}&5.57e+1&1.35e+1&1.94e-2&1.35e+1&1.98e-2\\ \hline
(0.05, 5.0e-4)&Linear&{\bf 1.05e+1}&2.59e+0&6.90e+0&9.25e-3&6.89e+0&9.69e-3\\ 
&Sigmoid&{\bf 1.03e+1}&4.78e+0&6.85e+0&1.20e-2&6.85e+0&1.23e-2\\ \hline
(0.05, 1.0e-3)&Linear&{\bf 1.04e+1}&2.56e+0&6.82e+0&9.17e-3&6.82e+0&9.52e-3\\ &Sigmoid&{\bf 1.02e+1}&1.21e+1&6.86e+0&1.34e-2&6.86e+0&1.38e-2\\ \hline
\end{tabular}
\end{center}
\caption{Results of real dataset simulation. Each result represents the average of $10^3$ simulation runs. The best value for each dataset in obj is in bold.
} \label{real2}
\end{table}

\section{Conclusion}
We formulated the problem of optimizing the price by taking into consideration the resulting bipartite graph matching, given the effect of the price on the probabilistic uncertainty in the graph. Our problem is a generalization of \citep{Tong2018DynamicPI} and can handle more practical situations in resource allocation platforms. We developed a $(1-1/e)$-approximation algorithm under the assumption that a convex min-cost flow problem can be solved exactly. 
We conducted simulation experiments on ride-sharing platforms and crowd-sourcing platforms, which showed that the proposed algorithm outputs more profitable solutions compared with the conventional methods in a practical amount of time.

Future work will include showing the effectiveness of the proposed algorithm by applying it to actual services, speeding it up, proving a tight approximation ratio for it, and developing an algorithm that allows Assumption~\ref{asp:average_rand} to be relaxed.

\section{Proof}
\subsection{Proof of Lemma~\ref{lem:inverse_concave}}
First, we will show that $\bar{\xi}_v'(x) < 0$ for all $x \in \Int(\mathcal X_v)$ and $v \in V$. Since $\bar{\xi}_v'(x)\le 0$ from Assumption~\ref{asp:average_rand}(a) and the definition \eqref{def:bar_xi} of $\bar{\xi}_v$, we will show $\bar{\xi}_v'(x)\neq 0$ for $x \in \Int(\mathcal X_v)$. We assume that there exist $b \in \Int(\mathcal X_v)$ which satisfy $\bar{\xi}_v'(b)=0$ for $v \in V$ and then derive a contradiction. From conditions (a) and (b) of Assumption \ref{asp:average_rand}, $p_v(x)>0$ for all $x \in \Int(\mathcal X_v)$. Then, $\bar{\xi}_v(x)=n_vp_v(x) > 0$ for all $x \in \Int(\mathcal X_v)$. Then, the following holds for arbitrary $x$ satisfying $x \le b$ and $x \in \Int (\mathcal X_v)$:
\begin{align}
0 \ge \bar{\xi}_v'(x)/\bar{\xi}_v(x) \ge \bar{\xi}_v'(b)/\bar{\xi}_v(b) =0,\label{eq:pvtx_0_0_2}
\end{align}
where the first inequality holds from $\bar{\xi}_v(x)>0$ and $\bar{\xi}_v'(x) \le 0$, and the second inequality holds since $\bar{\xi}_v'(x)/\bar{\xi}_v(x)$ is monotonically non-increasing from Assumption \ref{asp:average_rand}(c). From \eqref{eq:pvtx_0_0_2}, we have $\bar{\xi}_v'(x) = 0$ for any $x$ satisfying $x \le b$ and $x \in \Int (\mathcal X_v)$. This fact contradicts that $\bar{\xi}_v$ is bijective from Assumption \ref{asp:average_rand}. Therefore, $\bar{\xi}_v'(x) \neq 0$ for any $x \in \Int(\mathcal X_v)$, and thus $\bar{\xi}_v'(x) < 0$ for any $x \in \Int(\mathcal X_v)$.

Next, we show that $-\bar{\xi}_v^{-1}(z) z$ is convex w.r.t. $z \in \mathcal Z_v$ for all $v \in V$. For any $v \in V$, there exists only one $y$ that satisfies $z=\bar{\xi}_v(y)$ from Assumption~\ref{asp:average_rand} (a) when $z \in \mathcal Z_v$. Then, the following equality holds for any $z \in \Int (\mathcal Z_v)$:
\begin{align}
\left(\bar{\xi}_v^{-1}\left(z \right) z \right)' 
=\bar{\xi}_v^{-1}\left( z \right) + z \left( \bar{\xi}_v^{-1} (z) \right)' 
=y+\bar{\xi}_v(y)/\bar{\xi}'_v(y), \label{rev_proof}
\end{align}
where the second equality comes from $\bar{\xi}'_v(y) \neq 0$ for $y \in \Int(\mathcal X_v)$. Here, we consider that $z^1$ and $z^2$ satisfy $z^1 \le z^2$ and $z_1, z_2 \in \Int (\mathcal Z_v)$. Moreover, we let $y^1$ satisfy $z^1=\bar{\xi}_v(y^1)$ and $y^2$ satisfy $z^2=\bar{\xi}_v(y^2)$. Then, since $\bar{\xi}_v$ is monotonically decreasing from Assumption~\ref{asp:average_rand} (a), we have
\begin{align}
y^1 \ge y^2. \label{eq:y1_y2} 
\end{align}
Here, Assumption \ref{asp:average_rand} (c) and the definition \eqref{def:bar_xi} of $\bar{\xi}_v$ yield that $\bar{\xi}_v(y)/\bar{\xi}'_v(y)$ is monotonically non-decreasing for any $y\in \Int (\mathcal X_v)$. Then, $y+\bar{\xi}_v(y)/\bar{\xi}'_v(y)$ is monotonically non-decreasing for any $y \in \Int (\mathcal X_v)$. From this fact, \eqref{rev_proof}, and \eqref{eq:y1_y2}, we have $\left(\bar{\xi}_v^{-1}\left(z^1 \right) z^1 \right)' = y^1+\bar{\xi}_v(y^1)/\bar{\xi}'_v(y^1) \ge y^2+\bar{\xi}_v(y^2)/\bar{\xi}'_v(y^2) = \left(\bar{\xi}_v^{-1}\left(z^2 \right) z^2 \right)'$. Since $\left(\bar{\xi}_v^{-1}\left(z \right) z\right)'$ is monotonically decreasing for any interior point $z_1$ and $z_2$ of $\mathcal Z_v$, we find that $\bar{\xi}_v^{-1}\left(z\right) z$ is concave. This holds for all $v \in V$. $\Box$

\subsection{Proof of Theorem~\ref{thm:approximation_ratio} 
\label{prooftheorem1}}
First, we show that $\mathbb{E}_{\bmxi(\bx) \sim D(\bx)}[f(\bx,\bmxi)] \le \hat{f}(\bx)$. Consider a continuous relaxation of problem ${\rm (P_{sub})}$ by replacing $z_{(u,v)} \in \Z_{\ge0}$ with $z_{(u,v)} \ge 0$, which can be written as follows:
\begin{align}
{\rm (LP_p)}\ \ \max_{\bz} &\ \ \ \bh_1^\top \bz \nonumber \\
{\rm s.t. } &\ \ \ \bH_1 \bz \le \bH_2 \bmxi +\bh_2,\ \ 0 \le \bz, \nonumber
\end{align}
where $\bh_1$,\ $\bh_2$,\ $\bH_1$,\ $\bH_2$ are constant vectors and constant matrices. Since the feasible region of ${\rm (P_{sub})}$ is compact for any $\bmxi$ due to $\bz\in \Z_{\ge 0}^{E}$ and the second constraints of ${\rm (P_{sub})}$, the feasible region of ${\rm (LP_p)}$ is also compact. Then, the optimal value of ${\rm (LP_p)}$ exists since its objective value is continuous. Let $g(\bx,\bmxi)$ be the optimal value of ${\rm (LP_p)}$. Then, since $g(\bx,\bmxi)$ is equal to $f(\bx,\bmxi)$ from \citep[Theorem 5.12, Theorem~\ref{thm:approximation_ratio}1.2]{korte2012combinatorial}, we have
\begin{align}
 \mathbb{E}_{\bmxi\sim D(\bx)} [g(\bx,\bmxi)]=\mathbb{E}_{\bmxi\sim D(\bx)} [f(\bx,\bmxi)]. \label{eq:gx_equal_fx}
\end{align}
Here, the dual problem of ${\rm (LP_p)}$ is as follows:
\begin{align*}
{\rm (LP_d)}\ \ \min_{\by} &\ \ \ (\bH_2 \bmxi +\bh_2)^\top \by \nonumber \\
{\rm s.t. } &\ \ \ \bH_1^\top \by \ge \bh_1,\ \ 0 \le \by. \nonumber
\end{align*}
Since the dual problem has the same optimal value as the primal problem in the linear problem, the optimal value of ${\rm (LP_d)}$ is equal to ${\rm (LP_p)}$, that is, $g(\bx,\bmxi)$. Let $\bY$ be the feasible region of ${\rm (LP_d)}$. Since $(\bH_2 \bmxi +\bh_2)^\top \by$ is concave in $\bmxi$ for any $\by \in \bY$, we obtain $\min\nolimits_{\by \in \bY}(\bH_2 \bmxi +\bh_2)^\top \by$; that is, $g(\bx,\bmxi)$ is concave in $\bmxi$ from \citep[Section 3.2.3]{boyd_vandenberghe_2004}. Then, for any $\bx$, we have 
\begin{align*}
\mathbb{E}_{\bmxi\sim D(\bx)} [f(\bx,\bmxi)]=\mathbb{E}_{\bmxi\sim D(\bx)} [g(\bx,\bmxi)] \le g(\bx, \mathbb{E}_{\bmxi\sim D(\bx)}[\bmxi]) =\hat{f}(\bx), 
\end{align*}
where the equality comes from \eqref{eq:gx_equal_fx}, the first inequality holds from Jensen's inequality, and the last equality from the definition of $\hat{f}$ and $g$.

Next, we show that $(1-1/e) \hat{f}(\bx) \le \E_{\bmxi\sim D(\bx)}[f(\bx,\bmxi)] $. We give proofs for two cases: (i) the case where $\xi_v$ for all $v \in V$ follows a binomial distribution $\mathrm{Bin}(n_v, p_{v}(x_v))$ and (ii) the case where $\xi_v$ for all $v \in V$ follows a Poisson distribution $\Po(n_v p_v(x_v))$.

(i) Case where $\xi_v$ follows a binomial distribution $\mathrm{Bin}(n_v, p_{v}(x_v))$.\\ First, we describe Problem A and prove Lemma \ref{lem:approx_rate_brub}. We create $(n_{v_i}-1)$ duplicates for each node $v_i \in V$ and define a new set of nodes $V':=\bigcup_{{v_i} \in V} V(v_i)$, where $V(v_i):=\{v^1_i,v^2_i\dots,v^{n_{v_i}}_i\}$. For each node $u_i \in U$, we create $(c_{u_i}-1)$ duplicates and define a new set of nodes $U':=\bigcup_{{u_i} \in U} U(u_i)$, where $U(u_i):=\{u^1_i,u^2_i\dots,u_i^{c_{u_i}}\}$. Let $E':=\{ (u_i^q,v_j^r) \mid u_i \in U, v_j \in V, (u_i,v_j) \in E, q \in \{1 ,\dots, c_{u_i}\}, r \in \{1 ,\dots, n_{v_i}\} \}$. After that, we describe Problem A for given $\bx$.

\paragraph{Problem A.} We consider an undirected bipartite graph $G = (U',V',E')$ where each node $v_j^r \in V'$ is available with probability $p_{v_j}(x_{v_j})$. Suppose that the following steps (i) and (ii) are repeated until $U'$ or $V'$ becomes empty. (i) Choose $v^r_j \in V'$ and $u_i^q \in U'$ and try to match them, which is called {\it probing}. The probing succeeds with probability $p_{v_j}(x_{v_j})$ and it fails with probability $1-p_{v_j}(x_{v_j})$. (ii) If the probing succeeds, $v^r_j$ and $u_i^q$ are removed from $V'$ and $U'$, respectively, with the benefit of $(x_{v_j}+w_{(u_i,v_j)})$. If the probing fails, no profit is made and $v^r_j$ is removed from $V'$. Here, {\it what is the most profitable probing strategy?} \vspace{2mm}

\begin{lemma}
[{\cite[Theorem 2]{brubach2021improved}}] \label{lem:approx_rate_brub}
Let $\E[{\rm ALG}]$ be the expected profit obtained by the algorithm in \citep{brubach2021improved} for Problem A. Then, $\E[{\rm ALG}] \ge (1-1/e) \hat{f}(\bx)$. 
\end{lemma}

\begin{proof}
For problem \textbf{LP} defined in \citep{brubach2021improved}, let $V^{LP}:=U'\cup V'$, $E^{LP}:=E'$, $w^{LP}_{e}:=(x_{v_j}+w_{(u_i,v_j)})$ for all $e=(u_i^q,v_j^r) \in E'$, $t^{LP}_v:=1$ for all $v \in V' \subset V^{LP}$, $t^{LP}_u:=\infty$ for all $u \in U' \subset V^{LP}$, and $p^{LP}_{e}:=p_{v_j}(x_{v_j})$ for all $e \in \delta(v_j^r) \subseteq E^{LP}$. Under these settings, let $L^*$ be an optimal value of \textbf{LP}. Then, $\E[{\rm ALG}] \ge (1-1/e) L^*$ from \citep[Theorem~2]{brubach2021improved}. Since \textbf{LP} is equivalent to \eqref{low2} under the settings, $\E[{\rm ALG}] \ge (1-1/e) \hat{f}(\bx)$. 
\end{proof}

Next, we show $(1-1/e) \hat{f}(\bx) \le \E_{\bmxi\sim D(\bx)}[f(\bx,\bmxi)] $. From Lemma \ref{lem:approx_rate_brub}, if $\E_{\bmxi\sim D(\bx)}[f(\bx,\bmxi)] \ge \E[{\rm ALG}]$, we get $\E_{\bmxi\sim D(\bx)}[f(\bx,\bmxi)]\ge (1-1/e) \hat{f}(\bx)$. In Problem A, whether a matching involving $v \in V'$ succeeds or fails is determined regardless of other nodes' success and the order of $probings$. For a trial of Problem A, let $\bmu \in \{ 0,1 \}^{|V'|}$ be the variables that represent whether the matching is successful or not if each $v \in V'$ is chosen. Here, $\mu_{v}=1$ in the case of success, and $\mu_{v}=0$ in the case of failure. Then, let $\bz^\dagger(\bmu)$ be the matching done by the algorithm of \citep{brubach2021improved}. Here, $z_{(u,v)}^\dagger(\bmu)=1$ indicates that $u \in U'$ and $v \in V'$ are matched, and $z_{(u,v)}^\dagger(\bmu)=0$ indicates that they are not matched. Let $\hat{\bz}^\dagger$ and $\hat{\bmu}$ be the vectors summed over the replicated nodes, that is, $\hat{z}^\dagger_{(u_i,v_j)}= \sum_{q} \sum_{r} z_{(u_i^q,v_j^r)}^\dagger(\bmu)$ and $\hat{\mu}_{v_j}=\sum_{r} \mu_{v_j^r}$ for all $u_i \in U$ and $v_j \in V$. Here, the profit obtained by the algorithm is less than or equal to $f(\bx,\hat{\bmu})$ for any realization of $\bmu$: (i) $\hat{\bz}^\dagger$ is included in the feasible region of ${\rm (P_{sub})}$ with $\bmxi:=\hat{\bmu}$ from the setting of problem A; the profit obtained by the algorithm is $\sum_{(u,v)\in E} (x_v+w_{(u,v)}) \hat{z}^\dagger_{(u,v)}$, that is, the objective function of ${\rm (P_{sub})}$; $f(\bx,\hat{\bmu})$ is the optimal value of ${\rm (P_{sub})}$ with $\bmxi:=\hat{\bmu}$. Since the distribution of $\hat{\bmu}$ is the same as $D(\bx)$, we have $\E_{\hat{\bmu}\sim D(\bx)}[f(\bx,\hat{\bmu})] \ge\E[{\rm ALG}]$. Therefore, from Lemma \ref{lem:approx_rate_brub}, $\E_{\bmxi\sim D(\bx)}[f(\bx,\bmxi)] \ge (1-1/e) \hat{f}(\bx)$. 

(ii) Case where $\xi_v$ follows a Poisson distribution $\Po(n_v p_v(x_v))$.\\
Let $\Po(\bar{\bmxi}(\bm{x})):=\prod_{v \in V} \Po(\bar{\xi}_v(x_v))=\prod_{v \in V} \Po(n_v p_v(x_v))$. We will show that
\begin{align} \label{eq: theorem_2_Poisson}
 \prn*{1 - 1/e} \hat{f}(\bm{x}) \leq \mathbb{E}_{\bm{\xi} \sim \Po(\bar{\bmxi}(\bm{x}))} \bra{f(\bm{x}, \bm{\xi})} \leq \hat{f}(\bm{x}).
\end{align}
Let $(D^n(\bm{x}))_{n \in \Z_{> 0}}$ be a sequence of tuples of distributions, where $D^n(\bm{x}):= (D^n_v(x_v))_{v \in V}$ and $D^n_v(x_v) := \textrm{Bin} \prn*{n,\frac{\bar{\xi}_v(x_v)}{n}}$. A similar discussion as in (i) shows that the following holds for $n \in \Z_{> 0}$:
\begin{align*}
 \prn*{1 - \frac{1}{e}} \hat{f}(\bm{x}) \leq \mathbb{E}_{\bm{\xi} \sim D^n(\bm{x}) } \bra{f(\bm{x}, \bm{\xi})} \leq \hat{f}(\bm{x}).
\end{align*}
Therefore, if
\begin{align} \label{eq: lim}
 \lim_{n \to \infty} \mathbb{E}_{\bm{\xi} \sim D^n(\bm{x}) } \bra{f(\bm{x}, \bm{\xi})} = \mathbb{E}_{\bm{\xi} \sim \Po(\bar{\bmxi}(\bm{x}))} \bra{f(\bm{x}, \bm{\xi})},
\end{align}
then \cref{eq: theorem_2_Poisson} holds. Let $\bmxi^{\Po}$ be a random variable following $\Po(\bar{\bmxi}(\bm{x}))$, $\bmxi^{(n)}$ be a random variable following $D ^n(\bm{x})$, and $C \coloneqq \sum_{u \in U} c_u \max_{v \in N(u)} \abs*{w_{(u,v)} + x_v}$, where $N(u)$ is the set of nodes adjacent to $u$. Then, we have
\begin{align}
 &\abs*{
 \mathbb{E}_{\bm{\xi} \sim \Po(\bar{\bmxi}(\bm{x}))} \bra{f(\bm{x}, \bm{\xi})} - \mathbb{E}_{\bm{\xi} \sim D^n(\bm{x}) } \bra{f(\bm{x}, \bm{\xi})}
 } \nonumber
 \\
 &= \abs*{
 \sum_{\bm{\xi}' \in \Z_{\geq 0}^V } \Pr \bra*{\bm{\xi}^{\Po} = \bm{\xi'}}f(\bm{x}, \bm{\xi'}) - 
 \sum_{\bm{\xi}' \in \Z_{\geq 0}^V } \Pr \bra*{\bm{\xi}^{(n)} = \bm{\xi'}}f(\bm{x}, \bm{\xi'})
 } \nonumber
 \\
 &= \abs*{
 \sum_{\bm{\xi}' \in \Z_{\geq 0}^V } \prn*{\Pr \bra*{\bm{\xi}^{\Po} = \bm{\xi'}} - \Pr \bra*{\bm{\xi}^{(n)} = \bm{\xi'}}} f(\bm{x}, \bm{\xi'})
 } \nonumber
 \\ 
 &\leq 
 \sum_{\bm{\xi}' \in \Z_{\geq 0}^V } \abs*{\Pr \bra*{\bm{\xi}^{\Po} = \bm{\xi'}} - \Pr \bra*{\bm{\xi}^{(n)} = \bm{\xi'}}} \abs*{f(\bm{x}, \bm{\xi'})} \nonumber
 \\
 &\leq C \sum_{\bm{\xi}' \in \Z_{\geq 0}^V } \abs*{\Pr \bra*{\bm{\xi}^{\Po} = \bm{\xi'}} - \Pr \bra*{\bm{\xi}^{(n)} = \bm{\xi'}}} \label{eq: bound_by_C}
 \\ 
 &= C \sum_{\bm{\xi}' \in \Z_{\geq 0}^V } \abs*{ \prod_{i = 1}^{|V|} \Pr \bra*{\xi^{\Po}_i = \xi'_i} - \prod_{i = 1}^{|V|} \Pr \bra*{\xi^{(n)}_i = \xi'_i}} \nonumber
 \\ 
 &\leq C \sum_{\bm{\xi}' \in \Z_{\geq 0}^V } \sum_{i=1}^{|V|} \prn*{ \abs*{\Pr \bra*{\xi^{\Po}_i = \xi'_i} - \Pr \bra*{\xi^{(n)}_i = \xi'_i}} \prod_{j=1}^{i-1} \Pr \bra*{\xi^{\Po}_j = \xi'_j} \prod_{j=i+1}^{|V|} \Pr \bra*{\xi^{(n)}_j = \xi'_j}} \label{eq: triangle}
 \\
 &= C \sum_{i=1}^{|V|} 
 \prn*{\sum_{\xi'_i \in \Z_{\geq 0}} \abs*{\Pr \bra*{\xi^{\Po}_i = \xi'_i} - \Pr \bra*{\xi^{(n)}_i = \xi'_i}}}
 \prn*{\prod_{j=1}^{i-1} \sum_{\xi'_j \in \Z_{\geq 0}} \Pr \bra*{\xi^{\Po}_j = \xi'_j}}
 \prn*{\prod_{j=i+1}^{|V|} \sum_{\xi'_j \in \Z_{\geq 0}} \Pr \bra*{\xi^{(n)}_j = \xi'_j}} \nonumber
 \\
 &= C \sum_{i=1}^{|V|} 
 \prn*{\sum_{\xi'_i \in \Z_{\geq 0}} \abs*{\Pr \bra*{\xi^{\Po}_i = \xi'_i} - \Pr \bra*{\xi^{(n)}_i = \xi'_i}}}, \label{eq:po_bin_equiv}
 \end{align}
where the inequality \cref{eq: bound_by_C} holds since $|f(\bm{x}, \bm{\xi})| \leq |C|$ for arbitrary $\bm{\xi} \in \Z_{\geq 0}^V$ from the definition of ${\rm (P_{sub})}$. Moreover, \cref{eq: triangle} comes from the following lemma.
\begin{lemma}
If $a_1, \ldots, a_m$ and $b_1, \ldots, b_m$ is non-negative real value, then
 \begin{align}
 \abs*{\prod_{i=1}^m a_i - \prod_{i=1}^m b_i} \leq \sum_{i=1}^m \prn*{\abs*{a_i - b_i} \prn*{\prod_{j=1}^{i-1} a_j} \prn*{\prod_{j=i+1}^m b_j}}.
 \end{align}
\end{lemma}
\begin{proof}
Let $(c_i)_{i=0, \ldots, m}$ be the real number sequence such that $c_0 = \prod_{j=1}^m b_j, c_m = \prod_{j=1}^m a_j$, and $c_i = \prn*{\prod_{j=1}^i a_j} \prn*{\prod_{j=i+1}^m b_j}$ for $i=1,2,\dots,m-1$. Then, $c_i - c_{i-1} = (a_i - b_i) \prn*{\prod_{j=1}^{i-1} a_j} \prn*{\prod_{j=i+1}^m b_j}$ for $i=1, \ldots, m$. Therefore, 
\begin{align*}
 \abs*{\prod_{i=1}^m a_i - \prod_{i=1}^m b_i}
 &= \abs*{c_m - c_0} \\
 &\leq \sum_{i=1}^{m} \abs*{c_i - c_{i-1}} \\ 
 &= \sum_{i=1}^{m} \abs*{(a_i - b_i) \prn*{\prod_{j=1}^{i-1} a_j} \prn*{\prod_{j=i+1}^m b_j}} \\ 
 &= \sum_{i=1}^m \prn*{\abs*{a_i - b_i} \prn*{\prod_{j=1}^{i-1} a_j} \prn*{\prod_{j=i+1}^m b_j}},
\end{align*}
where the inequality comes from the triangle inequality.
\end{proof}
 
Here, the following theorem holds.
\begin{theorem}
[A special case of Le Cam's theorem \cite{le1960approximation}] \label{thm:lecams}
 \begin{align}
 \sum_{\xi_i \in \Z_{\geq 0}} \abs*{\Pr \bra*{\xi^{\Po}_i = \xi'_i} - \Pr \bra*{\xi^{(n)}_i = \xi'_i}} \leq \frac{2 \bar{\xi}_i(x_i)^2}{n}.
 \end{align}
\end{theorem}
From Theorem \ref{thm:lecams} and \eqref{eq:po_bin_equiv},
\begin{align*}
 \abs*{
 \mathbb{E}_{\bm{\xi} \sim \Po(\bar{\bmxi}(\bm{x}))} \bra{f(\bm{x}, \bm{\xi})} - \mathbb{E}_{\bm{\xi} \sim D^n(\bm{x}) } \bra{f(\bm{x}, \bm{\xi})}
 } 
 \leq 
 \frac{2 C \sum_{i=1}^{|V|}\bar{\xi}_i(x_i)^2}{n}.
\end{align*}
Therefore, when $n \to \infty$, it yields that $\mathbb{E}_{\bm{\xi} \sim D^n(\bm{x}) } \bra{f(\bm{x}, \bm{\xi})} \to \mathbb{E}_{\bm{\xi} \sim \Po(\bar{\bmxi}(\bm{x}))} \bra{f(\bm{x}, \bm{\xi})}$. 

\subsection{Proof of Lemma \ref{lem:opt_exist}} 
\label{prooflemma1}
To prove Lemma \ref{lem:opt_exist}, we will first prove Lemmas \ref{lemma: existence of M} and \ref{lemma: bounding region}.

\begin{lemma} \label{lemma: existence of M}
Suppose that Assumption \ref{asp:average_rand} holds. Then, the following holds for sufficiently large $M$ and all $v \in V$ such that $\mathcal X_v$ is unbounded from above:
\begin{enumerate}
\item[(i)] For arbitrary $M' > M$, $e=(u, v) \in E$, and $\ell \in \{0, \max_{v \in V} \bar{\xi}^{-1}_{v} \left(\frac{1}{|E|} \right) + \max_{e \in E} w_{e} \}$, the following holds:
\begin{align}
 (M' + w_{e} - \ell)\bar{\xi}_v(M') - (M + w_{e} - \ell)\bar{\xi}_v(M) < 0. \label{eq: 1st condition}
\end{align}

\item[(ii)] 
\begin{align}
\bar{\xi}_v(M) \leq \frac{1}{|E|}. \label{eq: 2nd condition}
\end{align}

\end{enumerate}
\end{lemma}

\begin{proof}
From the second case of Assumption \ref{asp:average_rand} (b) and the definition \eqref{def:bar_xi} of $\bar{\xi}_v$, we have $\lim_{x \to \infty} \bar{\xi}_v(x) = 0$. Therefore, condition (ii) is satisfied for sufficiently large $M$. Lemma \ref{lemma: existence of M} holds if we show that condition (i) is satisfied when $M$ is sufficiently large.

Here, let $h_{cv}(x) \coloneqq (x+c)\bar{\xi}_v(x)$ for $c \in \mathbb{R}$. Then, for $x \in \Int(\mathcal X_v)$,
\begin{align}
 h'_{cv}(x) = \bar{\xi}_v(x) + (x+c)\bar{\xi}'_v(x) 
 = \bar{\xi}'_v(x) \left( \frac{\bar{\xi}_v(x)}{\bar{\xi}'_v(x)} + x+c \right),
\end{align}
where the last equality holds since $\bar{\xi}_v'(x) < 0$ (that is, $\bar{\xi}_v'(x) \neq 0$) from Lemma~\ref{lem:inverse_concave}. From Assumption~\ref{asp:average_rand}(c) and the definition \eqref{def:bar_xi} of $\bar{\xi}_v$, we have ${\bar{\xi}_v(x)}/{\bar{\xi}'_v(x)}$ is monotonically non-decreasing. If we take a sufficiently large $M$, then ${\bar{\xi}_v(x)}/{\bar{\xi}'_v(x)} + x + c > 0$ for any $x > M$. Therefore, since $\bar{\xi}_v'(x) < 0$, it yields $h_{cv}'(x) < 0$ for $x > M$; i.e., $h_{cv}(x)$ is monotonically decreasing for $x > M$. Hence, when $M$ is sufficiently large, $h_{cv}(M')-h_{cv}(M)<0$ for arbitrary $M'>M$. Then, we see that \eqref{eq: 1st condition} holds by letting $c := w_e - \ell$.
\end{proof}

Next, we prove Lemma \ref{lemma: bounding region}. This lemma shows that any feasible solution of (PA) can be modified by Algorithm \ref{alg:improve solution} to a feasible solution such that $x_i$ for given $i$ is contained in a bounded closed interval without loss of the objective value. 
We use Algorithm \ref{alg:improve solution} for all $i \in V$ to show that any feasible solution of (PA) can be modified into a feasible solution contained in a compact set without loss of the objective value. This fact leads to the existence of an optimal solution of (PA).

\begin{lemma} \label{lemma: bounding region}
Suppose that Assumption \ref{asp:average_rand} holds. Let $M$ be a constant that satisfies the conditions of Lemma~\ref{lemma: existence of M}. Let $L:=\max_{e \in E} w_e$. If Algorithm \ref{alg:improve solution} is applied to a feasible solution $(\bx, \bz)$ for (PA), then there exists edge $g$ satisfying the condition on line \ref{state:choose_gr} when the condition on line \ref{if:z_1} holds. Moreover, the following hold for the output $(\hat{\bx}, \hat{\bz})$ of Algorithm~\ref{alg:improve solution}:
\begin{enumerate}
\item[(I)] The output $(\hat{\bx}, \hat{\bz})$ is a feasible solution for (PA).
\item[(II)] The objective value of (PA) of the output $(\hat{\bx}, \hat{\bz})$ is greater than or equal to that of the input $(\bx, \bz)$.
\item[(III)] If $\mathcal X_i$ is unbounded above, then $\hat{x}_{v} = x_v$ for all $v \in V \setminus \{i\}$ and 
 \begin{align*}
 \hat{x}_i = 
 \begin{cases}
 x_{i}, & \text{if} \quad -L \leq x_{i} \leq M, \\
 -L, & \text{if} \quad x_{i} < -L, \\
 M, & \text{if} \quad x_{i} > M.
 \end{cases}
 \end{align*}
\item[(IV)] If $\mathcal X_i$ is bounded above, then $\hat{x}_{v} = x_v$ for all $v \in V \setminus \{i\}$ and 
 \begin{align*}
 \hat{x}_i = 
 \begin{cases}
 x_{i}, & \text{if} \quad -L \leq x_{i}, \\
 -L, & \text{if} \quad x_{i} < -L.
 \end{cases}
 \end{align*}
\end{enumerate}
\end{lemma}

\begin{algorithm}[t]
  \caption{$\mathrm{IMPROVE}$}
  \label{alg:improve solution}
  \begin{algorithmic}[1]
    \Require{$\bx \in \mathbb{R}^{V}, \bz \in \mathbb{R}^{E}, i \in V$}
    \If{$x_i < -L$}
        \State{$\hat{x}_{v} \leftarrow 
        \begin{cases}
        -L \ &\text{if}\ v=i,\\
        x_v \ &\text{otherwise,}
        \end{cases}$ \quad \text{for $v \in V$}}\label{state:x_to_mL}
        \State{$\hat{z}_{e} \leftarrow 
        \begin{cases}
        0 \ &\text{if}\ e \in \delta(i), \\
        z_{e} \ & \text{otherwise,}
        \end{cases}$ \quad \text{for $e \in E$}} \label{state:z_to_0}
    \ElsIf{$\mathcal X_i$ is unbounded from above and $x_{i} > M$} \label{con:x>M}
        \State{$\hat{x}_v \leftarrow 
        \begin{cases}
        M \ &\text{if}\ v=i,\\
        x_v \ & \text{otherwise,}
        \end{cases}$  \quad \text{for $v \in V$}} \label{state:x_to_M}
        \State{Select an edge, $e'=(u_{e'},i) \in \delta(i)$} \label{state:select_edge}
        \State{$\hat{z}_{e}\leftarrow
        \begin{cases}
        z_{e} + \bar{\xi}_i(M) - \bar{\xi}_i(x_{i}) \ &\text{if}\ e=e',\\
        z_{e} \ & \text{otherwise,}
        \end{cases}$ \quad \text{for $e \in E$}}\label{state:z_to_pM_px}
         \If{$\tsum_{e \in \delta(u_{e'})} \hat{z}_{e} > c_{u_{e'}}$}\label{if:z_1}
           \State{Select $g \in \delta(u_{e'})$ satisfying $\hat{z}_{g} > \frac{1}{|E|}$}\label{state:choose_gr}
            \State{$\hat{z}_{g} \leftarrow \hat{z}_{g} - \bar{\xi}_i(M) + \bar{\xi}_i(x_{i})$}\label{state:z_gr_to_pM_px}
        \EndIf
    \Else \label{con:-L<x<M}
        \State{$\hat{\bx} \leftarrow \bx$, $\hat{\bz} \leftarrow \bz$} \label{no_change}
    \EndIf
    \State{\textbf{return} $\hat{\bx}, \hat{\bz}$}
  \end{algorithmic}
 \end{algorithm}

\begin{proof}
When Algorithm \ref{alg:improve solution} is applied to a feasible solution $(\bx, \bz)$ for (PA), there exists an edge $g \in E$ satisfying $\hat{z}_g > \frac{1}{|E|}$ in line~\ref{state:choose_gr}. 
If $\hat{z}_e \le \frac{1}{|E|}$ for all $e \in \delta(u_{e'})$, then the condition on line 8 is not satisfied because $\sum_{e \in \delta(u_{e'})} \hat{z}_e \le \sum_{e \in \delta(u_{e'})}\frac{1}{|E|} \le 1 \le c_{u_{e'}}$.

Then, before showing conditions (I)-(III), we present some facts. Since $(\bx,\bz)$ is a feasible solution for (PA), we have
\begin{align}
&\tsum_{e \in \delta(v)} z_{e} \le \bar{\xi}_v(x_v), \quad \forall v \in V, \label{ineq:px_z_feasible}\\
&\tsum_{e \in \delta(u)} z_e \le c_u, \quad \forall u \in U, \label{ineq:u_t_feasible}\\
&0 \le \bz. \label{ineq:px_z_in_0_1}
\end{align}

Here, we consider the case where $x_{i}>M$ and $\mathcal X_i$ is unbounded from above. Then, Algorithm \ref{alg:improve solution} performs lines \ref{state:x_to_M}--\ref{state:z_to_pM_px}. At the end of line \ref{state:z_to_pM_px},
\begin{align*}
\tsum_{e \in \delta(i)} \hat{z}_{e} = \tsum_{e \in \delta(i)} z_{e} + \bar{\xi}_i(M) - \bar{\xi}_i(x_{i}) \le \bar{\xi}_i(x_{i}) +\bar{\xi}_i(M) - \bar{\xi}_i(x_{i}) = \bar{\xi}_i(M) \le \frac{1}{|E|},
\end{align*}
where the first inequality comes from \eqref{ineq:px_z_feasible} and the last inequality comes from \eqref{eq: 2nd condition}. Then, there does not exist $\hat{z}_{e}$ satisfying $\hat{z}_{e} > \frac{1}{|E|}$ for $e \in \delta(i)$. Therefore, when line \ref{state:choose_gr} is executed,
\begin{align} \label{eq:g_set}
g \notin \delta(i), {\rm and \ then,}\ g \neq e',
\end{align}
where $g$ is the chosen edge in line \ref{state:choose_gr} and $e'$ is the chosen edge in line \ref{state:select_edge}.

In the following, we prove that conditions (I), (II), (III), and (IV) hold.

\paragraph{Proof of condition (I):}
At the end of Algorithm \ref{alg:improve solution}, condition (I) is satisfied if the following hold:
\begin{itemize}
\item[(i)] $ \tsum_{e \in \delta(v)} \hat{z}_{e} \le \bar{\xi}_v(\hat{x}_{v})$ for all $v \in V,$
\item[(ii)] $\tsum_{e \in \delta(u)} \hat{z}_{e} \le c_u$ for all $u \in U$, and
\item[(iii)] $0\le \hat{\bz}.$
\end{itemize}
Here, if ``$-L \le x_{i} \le M$ and $\mathcal X_i$ is unbounded from above'' or ``$-L \le x_{i}$ and $\mathcal X_i$ is bounded from above'', then $(\hat{\bx}, \hat{\bz}) = (\bx, \bz)$ from Algorithm \ref{alg:improve solution}. Therefore, condition (I) holds since $(\bx, \bz)$ is a feasible solution for (PA). Then, we show that (i), (ii), and (iii) hold in the following two cases: \textbf{(a)} $x_{i}<-L$; \textbf{(b)} $x_{i}>M$ and $\mathcal X_i$ is unbounded from above.

\subparagraph{(a)} Case where $x_i<-L$:\\
\textbf{(i)} From line~\ref{state:x_to_mL}--\ref{state:z_to_0} of Algorithm \ref{alg:improve solution} and \eqref{ineq:px_z_feasible}, $\bar{\xi}_v(\hat{x}_{v})=\bar{\xi}_v(x_v) \ge \tsum_{e \in \delta(v)} z_e = \tsum_{e \in \delta(v)} \hat{z}_e$ for all $v\in V \setminus \{i\}$ and $\bar{\xi}_i(\hat{x}_i)=\bar{\xi}_i(-L) \ge 0= \tsum_{e \in \delta(i)} \hat{z}_e$. Here, we already assumed that there exists $x' \in \mathcal X_v$ such that $-L\le x'$ for all $v \in V$ without loss of generality in Section \ref{sec:optimization_problem}. Therefore, $-L \in \mathcal X_i$ since $x_i \in \mathcal X_i$, $x' \in \mathcal X_i$, $x_i \le -L \le x'$, and $\mathcal X_i$ is an interval.
\\
\textbf{(ii)} Line~\ref{state:z_to_0} does not increase $z_e$ for all $e \in E$ since $\bz \ge 0$ from \eqref{ineq:px_z_in_0_1}. Therefore, for all $u \in U$,
\begin{align*}
\tsum_{e \in \delta(u)} \hat{z}_e
\le \tsum_{e \in \delta(u)} z_e 
\le c_u.
\end{align*}
Here, the second inequality holds from \eqref{ineq:u_t_feasible}.\\
\textbf{(iii)} From line~\ref{state:z_to_0}, $\hat{z}_e=0$ for all $e \in \delta(i)$. From \eqref{ineq:px_z_in_0_1}, $\hat{z}_e=z_e \ge 0$ for all $e \in E \setminus \delta(i)$. Therefore, $0\le \hat{\bz}$.

\subparagraph{(b)} Case where $x_{i}>M$ and $\mathcal X_i$ is unbounded from above:\\
\textbf{(i)} For $(\hat{\bx},\hat{\bz})$ at the end of line~\ref{state:z_to_pM_px}, the following holds:
\begin{align*}
\bar{\xi}_i(\hat{x}_i)=\bar{\xi}_i(M) \ge \bar{\xi}_i(M) - \bar{\xi}_i(x_i)+\tsum_{e \in \delta(i)} z_{e} 
= \tsum_{e \in \delta(i)} \hat{z}_e,
\end{align*}
where the inequality holds from \eqref{ineq:px_z_feasible}. Moreover, for all $v \in V\setminus \{i\}$, we have that $\bar{\xi}_v(\hat{x}_v) = \bar{\xi}_v(x_v) \ge \tsum_{e \in \delta(v)} z_e = \tsum_{e \in \delta(v)}\hat{z}_e$. Therefore, at the end of line~\ref{state:z_to_pM_px}, $\bar{\xi}_v(\hat{x}_v) \ge \tsum_{e \in \delta(v)} \hat{z}_e$ for all $v \in V$. Line~\ref{state:z_gr_to_pM_px} does not increase $\hat{z}_{g}$ since $- \bar{\xi}_i(M)+ \bar{\xi}_i(x_{i}) \le 0$. Thus, at the end of Algorithm \ref{alg:improve solution}, $\bar{\xi}_v(\hat{x}_{v}) \ge \tsum_{e \in \delta(v)}\hat{z}_{e}$ for all $v \in V$.

\noindent \textbf{(ii)} For all $u \in U \setminus \{u_{e'}\}$, it follows from \eqref{ineq:u_t_feasible} that 
\begin{align*}
 \tsum_{e \in \delta(u)} \hat{z}_e 
= \tsum_{e \in \delta(u)} z_e \le c_u.
\end{align*}
When the condition on line \ref{if:z_1} is not satisfied, $\tsum_{e \in \delta(u_{e'})} \hat{z}_e \le c_{u_{e'}}$. When the condition on line \ref{if:z_1} is satisfied,
\begin{align*}
 \tsum_{e \in \delta(u_{e'})} \hat{z}_e 
= \tsum_{e \in \delta(u_{e'})\setminus \{e',g\}} z_e + z_{e'} + \bar{\xi}_i(M) - \bar{\xi}_i(x_{i}) + z_g - \bar{\xi}_i(M)+ \bar{\xi}_i(x_{i}) = \tsum_{e \in \delta(u_{e'})} z_e 
 \le c_{u_{e'}}.
\end{align*}
Note that $g \neq e'$ from \eqref{eq:g_set}.

\noindent \textbf{(iii)} Since line~\ref{state:z_to_pM_px} only raises $z_{e'}$ by $\bar{\xi}_i(M)-\bar{\xi}_i(x_i)>0$, $\hat{\bz} \ge 0$ from \eqref{ineq:px_z_in_0_1} at the end of line~\ref{state:z_to_pM_px}. Since $\hat{z}_{g} > \frac{1}{|E|}$ at the end of line \ref{state:choose_gr}, at the end of line \ref{state:z_gr_to_pM_px},
\begin{align*}
\hat{z}_{g}>\frac{1}{|E|}- \bar{\xi}_i(M) + \bar{\xi}_i(x_i) > \frac{1}{|E|} -\bar{\xi}_i(M) \ge \frac{1}{|E|} - \frac{1}{|E|} =0.
\end{align*}
The second inequality holds since $\bar{\xi}_i(x_i) > 0$. The third inequality comes from \eqref{eq: 2nd condition}. Therefore, $0\le \hat{\bz}$ at the end of Algorithm \ref{alg:improve solution}.

Consequently, in both cases (a) and (b), the output $(\hat{\bx},\hat{\bz})$ of Algorithm \ref{alg:improve solution} satisfies condition (I).

\paragraph{Proof of condition (II):} if $-L \le x_{i} \le M$ and $\mathcal X_i$ is unbounded from above or $-L \le x_{i}$ and $\mathcal X_i$ is bounded from above, then condition (II) holds since $(\bx,\bz)=(\hat{\bx},\hat{\bz})$ from Algorithm~\ref{alg:improve solution}. Therefore, we will show that condition (II) holds in two cases: \textbf{(a)} $x_i<-L$; \textbf{(b)} $x_i>M$ and $\mathcal X_i$ is unbounded from above.

\textbf{(a)} Case where $x_i<-L$:\\
Line~\ref{state:x_to_mL}--\ref{state:z_to_0} increases the objective value by 
\begin{align*}
- \tsum_{e \in \delta(i)} (w_e + x_i)z_{e} \ge -\tsum_{e \in \delta(i)} (\max_{e \in E} w_e-L)z_{e} = 0.
\end{align*}
Here, the inequality follows from \eqref{ineq:px_z_in_0_1} and $x_i<-L$, and the equality follows from the definition of $L$.

\textbf{(b)} Case where $x_i>M$ and $\mathcal X_i$ is unbounded from above:\\ Suppose that the condition on line~\ref{if:z_1} is not satisfied. Then, lines~\ref{state:x_to_M}-\ref{state:z_to_pM_px} increase the objective value by
\begin{align*}
&\tsum_{e \in \delta(i)\setminus \{e'\}} (w_e+M) z_{e} + (w_{e'}+M) (z_{e'} + \bar{\xi}_i(M)-\bar{\xi}_i(x_i)) -
\tsum_{e \in \delta(i)} (w_e+x_i) z_{e} \\
&= (M-x_i) \tsum_{e \in \delta(i)} z_{e}
+ (w_{e'}+M)(\bar{\xi}_i(M)-\bar{\xi}_i(x_i)) \\
&\ge (M-x_i)\bar{\xi}_i(x_i) + (w_{e'}+M)(\bar{\xi}_i(M)-\bar{\xi}_i(x_i)) \\
& = (M+w_{e'})\bar{\xi}_i(M) - (x_i+w_{e'}) \bar{\xi}_i(x_i) \\
& > 0.
\end{align*}
The first inequality comes from \eqref{ineq:px_z_feasible} and $M-x_i<0$, while the second inequality comes from \eqref{eq: 1st condition}, where $\ell=0$. Therefore, condition (II) is satisfied when the condition on line~\ref{if:z_1} doesn't hold. Next, suppose that the condition on line~\ref{if:z_1} is satisfied. Let $v(g)$ be $v \in V$ incident to $g$. At the end of line~\ref{state:choose_gr}, 
\begin{align}
 \bar{\xi}_{v(g)}(x_{v(g)}) \ge \tsum_{e \in \delta(v(g))} z_{e} \ge z_{g} = \hat{z}_{g} > \frac{1}{|E|}, \label{eq:x_vg_small}
\end{align}
where the first inequality comes from \eqref{ineq:px_z_feasible}, the equality from \eqref{eq:g_set} and the operation of line~\ref{state:z_to_pM_px}, and the last inequality from the condition on line~\ref{state:choose_gr}. Let $\ell:=\max_{v \in V} \bar{\xi}_v^{-1}\left(\frac{1}{|E|}\right) + \max_{e\in E} w_e$. Then,
\begin{align}
\ell = \max_{v \in V} 
\bar{\xi}_v^{-1}\left(\frac{1} {|E|}\right) + \max_{e\in E} w_e \ge \bar{\xi}_{v(g)}^{-1}\left(\frac{1}{|E|}\right) + w_g \ge x_{v(g)} + w_g. \label{ineq:p_vgtrt_ge_x_w}
\end{align}
The second inequality holds from \eqref{eq:x_vg_small} since $\bar{\xi}_v^{-1}$ is monotone decreasing for all $v \in V$ from Assumption \ref{asp:average_rand} (a). Then, lines~\ref{state:x_to_M}-\ref{state:z_to_pM_px} and \ref{if:z_1}-\ref{state:z_gr_to_pM_px} increase the objective value by
\small
\begin{align*}
&\tsum_{e \in \delta(i)\setminus \{e'\}} (w_e+M) z_{e} + (w_{e'}+M) (z_{e'} + \bar{\xi}_i(M)-\bar{\xi}_i(x_i)) -
\tsum_{e \in \delta(i)} (w_e+x_i) z_{e} + (w_{g}+x_{v(g)}) (-\bar{\xi}_i(M)+\bar{\xi}_i(x_i)) \\
&= (M-x_i) \tsum_{e \in \delta(i)} z_e
+ (w_{e'}+M)(\bar{\xi}_i(M)-\bar{\xi}_i(x_i)) - (w_g+x_{v_g}) (\bar{\xi}_i(M)-\bar{\xi}_i(x_i)) \\
&\ge (M-x_i) \tsum_{e \in \delta(i)} z_e
+ (w_{e'}+M)(\bar{\xi}_i(M)-\bar{\xi}_i(x_i)) - \ell (\bar{\xi}_i(M)-\bar{\xi}_i(x_i)) \\
&\ge (M-x_i)\bar{\xi}_i(x_i) + (w_{e'}+M)(\bar{\xi}_i(M)-\bar{\xi}_i(x_i)) - \ell (\bar{\xi}_i(M)-\bar{\xi}_i(x_i)) \\
& = (M+w_{e'} - \ell)\bar{\xi}_i(M) - (x_i+w_{e'}- \ell) \bar{\xi}_i(x_i) \\
& > 0.
\end{align*}
\normalsize
The first inequality holds from $\bar{\xi}_i(M)-\bar{\xi}_i(x_i)>0$ and \eqref{ineq:p_vgtrt_ge_x_w}. The second inequality holds from \eqref{ineq:px_z_feasible} and $M-x_i<0$. The third inequality comes from \eqref{eq: 1st condition}. Therefore, condition (II) is satisfied.

Consequently, in both cases (a) and (b), the output $(\hat{\bx},\hat{\bz})$ of Algorithm \ref{alg:improve solution} satisfies condition (II).

\paragraph{Proof of condition (III):} It is clearly satisfied from line~\ref{state:x_to_mL}, line~\ref{state:x_to_M}, and line~\ref{no_change} in Algorithm~\ref{alg:improve solution}.

\paragraph{Proof of condition (IV):} It is clearly satisfied from line~\ref{state:x_to_mL} and line~\ref{no_change} in Algorithm~\ref{alg:improve solution}.
\end{proof}

Next, we use Lemma \ref{lemma: bounding region} to prove Lemma \ref{lem:opt_exist}.

\begin{proof}
Let $(\hat{\bx}, \hat{\bz})$ be the output of Algorithm~\ref{alg:generate bounded solution} for an arbitrary feasible solution $(\bx, \bz)$ of (PA). Moreover, let us denote
\begin{align}
 V_1 \coloneqq \{v \in V \mid \mathcal X_v \textrm{ is bounded from above} \} \textrm{ and } V_2 \coloneqq \{v \in V \mid \mathcal X_v \textrm{ is unbounded from above} \}. \label{def:V_1_V_2}
\end{align}
Then, from Lemma \ref{lemma: bounding region}, $(\hat{\bx}, \hat{\bz})$ satisfies the following conditions: (I) $(\hat{\bx}, \hat{\bz})$ is a feasible solution of (PA), (II) the objective value of (PA) of $(\hat{\bx}, \hat{\bz})$ is greater than or equal to that of $(\bx, \bz)$, (III) $-L \leq \hx_{v} \leq M$ for all $v \in V_2$, and (IV) $-L \leq \hx_{v}$ for all $v \in V_1$. Therefore, (PA) with the constraints $-L \leq x_v$ for all $v \in V_1$ and $-L \leq x_v \leq M$ for all $v \in V_2$ is equivalent to (PA). Here, the problem with the additional constraints has an optimal solution from the extreme-value theorem \cite[Section 1.6]{royden1988real} since it has a non-empty bounded closed set as the feasible region and a continuous function as the objective function. Therefore, (PA) has an optimal solution.
\end{proof}

\begin{algorithm}[t]
  \caption{Generate bounded solution}
  \label{alg:generate bounded solution}
  \begin{algorithmic}[1]
    \Require{$\bx \in \mathbb{R}^{V}, \bz \in \mathbb{R}^{E}$}
    \State{$\hbx \leftarrow \bx, \hbz \leftarrow \bz$}
    \For{$i \in V$}
        \State{$\hbx, \hbz \leftarrow \mathrm{IMPROVE}(\hbx, \hbz, i)$}
    \EndFor
    \State{\textbf{return} $\hbx, \hbz$}
  \end{algorithmic}
\end{algorithm}

\subsection{Proof of Theorem~\ref{thm:pa_p_solution} \label{prooftheorem2}}
Let $(\hat{\bx},\ \hat{\bz})$ be an optimal solution for (PA). First, we show that $\sum_{e \in \delta (v)} \hat{z}_e \in \mathcal Z^n_v$ for all $v \in V$, where $\mathcal Z^n_v:=\{n_v z \mid z \in \mathcal Z_v \}$. Note that $\mathcal Z^n_v$ is the range of the function $\bar{\xi}_v$ from the definition \eqref{def:bar_xi}. For each ${v} \in V$, we consider two cases according to whether $\sum_{e \in \delta ({v})} \hat{z}_e >0$ or $\sum_{e \in \delta ({v})} \hat{z}_e =0$. 

(i) Case where $\sum_{e \in \delta ({v})} \hat{z}_e >0$. \\ Since $\mathcal Z_{{v}}$ is the image of the connected set $\mathcal X_{{v}}$ by the continuous function $p_{{v}}$, $\mathcal Z_{{v}}$ is a connected set. Then, $\mathcal Z^n_{{v}}$ is a connected set from its definition. From the first constraints of (PA), $\sum_{e \in \delta ({{v}})} \hat{z}_e\le c$, where $c:=\bar{\xi}_{{v}}(\hat{x}_{{v}})$. Then, since $\lim_{x \rightarrow \infty} \bar{\xi}_{{v}}(x) = 0$ or $\bar{\xi}_{{v}}(\max {\mathcal X}_v) = 0$ from \mbox{Assumption~\ref{asp:average_rand}}(b) and the definition \eqref{def:bar_xi} of $\xi_v$, there exists a real number $d \in \mathcal Z^n_{{v}}$ such that $0 \le d < \sum_{e \in \delta ({{v}})} \hat{z}_e$. Since $\mathcal Z^n_{{v}}$ is a connected set and $d < \sum_{e \in \delta ({{v}})} \hat{z}_e\le c$ for $c, d \in \mathcal Z^n_{{v}}$, we get $\sum_{e \in \delta ({{v}})} \hat{z}_e \in \mathcal Z^n_{{v}}$. 

(ii) Case where $\sum_{e \in \delta ({{v}})} \hat{z}_e = 0$. \\ 
We will show that $0 \in \mathcal Z^n_{{v}}$, which yields $\sum_{e \in \delta (v)} \hat{z}_e \in \mathcal Z^n_v$. In particular, we will assume that $0 \notin \mathcal Z^n_{{v}}$ and obtain a contradiction. We pick an arbitrary node ${u}\in U$, which is adjacent to ${{v}}$. Note that we have assumed $\delta(v) \neq \emptyset$ in Section \ref{sec:optimization_problem} without loss of generality. Here, $\sum_{e \in \delta ({u})} \hat{z}_e \le c_{{u}}$ from the constraints of (PA). We will consider two cases. 

{(ii-a)} Case where $\sum_{e \in \delta ({u})} \hat{z}_e <c_{{u}}$. \\ 
Pick a number $x_M \in \{ x \mid x+w_{({u}, {v})}>0, x \in \mathcal X_v\}$. 
Note that we have assumed such $x_M$ exists without loss of generality in Section \ref{sec:optimization_problem}. 
Since $0 \notin \mathcal Z^n_v$, we have $\bar{\xi}_v(x)>0$ for all $x \in \mathcal X_v$. 
Then, there exists $\epsilon$ satisfying $0 <\epsilon \le \bar{\xi}_v(x_M)$ and $\sum_{e \in \delta ({u})} \hat{z}_e +\epsilon \le c_{{u}}$. Replacing $\hat{x}_{{v}}$ with $x_M$ and $\hat{z}_{({u}, {v})}(=0)$ with $\epsilon$ increases the objective value for (PA) because $(x_M+w_{({v},{u})}) \epsilon \ >0$ and this modification does not impair feasibility. This contradicts the optimality of $(\hat{\bx}, \hat{\bz})$ for (PA). 

{(ii-b)} Case where $\sum_{e \in \delta ({u})} \hat{z}_e =c_{{u}}$. \\ 
There exists $\bar{v} \ (\neq {v}) \in V$ satisfying $z_{({u}, \bar{v})}>0$. 
Since $\lim_{x \rightarrow \infty} \bar{\xi}_{{v}}(x)=0$ from $0 \notin \mathcal Z^n_{{v}}$, Assumption~\ref{asp:average_rand}(b), and the definition \eqref{def:bar_xi} of $\bar{\xi}$, there exists $x_M$ satisfying $x_M+w_{({u}, {v}) }>\hat{x}_{\bar{v}}+w_{({u}, \bar{v})}$ and $0<\bar{\xi}_{{v}}(x_M)=\epsilon < z_{({u}, \bar{v})}$. 
Replacing $\hat{x}_{{v}}$ with $x_M$, $\hat{z}_{(u,v)} \ (=0)$ with $\epsilon$, and $\hat{z}_{({u}, \bar{v})}$ with $\hat{z}_{({u}, \bar{v})}-\epsilon$ increases the objective value for (PA) because $(x_M+w_{({u},{v})}) \epsilon - (\hat{x}_{\bar{v}}+w_{({u}, \bar{v} )}) \epsilon >0$. This modification does not impair feasibility. This contradicts the optimality of $(\hat{\bx}, \hat{\bz})$ for (PA). \\ 
From (i) and (ii), we get $\sum_{e \in \delta (v)} \hat{z}_e \in \mathcal Z^n_v$ for all $v \in V$.

Next, we show that there exists an optimal solution $(\tilde{\bx},\ \tilde{\bz})$ for (PA) that satisfies $\bar{\xi}_v(\tilde{x}_{v}) = \sum\nolimits_{e \in \delta(v)} \tilde{z}_e$ for all $v$. Let $(\hat{\bx}, \hat{\bz})$ be an optimal solution for (PA). Then, from the constraints of (PA), $\sum\nolimits_{e \in \delta(v)} \hat{z}_{e} \le \bar{\xi}_v(\hat{x}_v)$ for all $v \in V$. Suppose that there exists $v$ satisfying $\sum\nolimits_{e \in \delta(v)} \hat{z}_{e} < \bar{\xi}_v(\hat{x}_v)$ and let $\hat{V}:=\{ v \mid \sum\nolimits_{e \in \delta(v)} \hat{z}_e < \bar{\xi}_v(\hat{x}_v) \}$. Since $\bar{\xi}_v$ is strictly decreasing and $\sum_{e \in \delta (v)} \hat{z}_e \in \mathcal Z^n_v$, there exists a positive scalar $d_v$ that satisfies $\sum\nolimits_{e\in \delta(v)} \hat{z}_e = \bar{\xi}_v(\hat{x}_v+d_v)$ for all $v \in \hat{V}$. Let $\tilde{x}_v$ be $\hat{x}_v+d_v$ for all $v \in \hat{V}$, $\tilde{x}_v$ be $\hat{x}_v$ for all $v \notin \hat{V}$, and $\tilde{\bz}$ be $\hat{\bz}$. Then, $(\tilde{\bx},\ \tilde{\bz})$ is an optimal solution for (PA) because it is a feasible solution and the objective value is greater than or equal to the optimal value from $\hat{\bz} \ge 0$. Therefore, in (PA), there exists an optimal solution $(\tilde{\bx},\tilde{\bz})$ satisfying $\sum\nolimits_{e\in \delta(v)} \tilde{z}_e = \bar{\xi}_v(\tilde{x}_v)$ for all $v \in V$. Hence, (PA) with the constraint, $\sum\nolimits_{e \in \delta(v)} z_e = \bar{\xi}_v(x_v)$ for all $v \in V$, is equivalent to (PA). Since $\bar{\xi}_v$ is a bijective function from Assumption~\ref{asp:average_rand}(a), $x_v=\bar{\xi}_v^{-1} \big( \sum\nolimits_{e \in \delta(v)} z_e \big)$ when $\sum\nolimits_{e \in \delta(v)} z_e = \bar{\xi}_v(x_v)$. By substituting $x_v=\bar{\xi}_v^{-1} \big( \sum\nolimits_{e \in \delta(v)} z_e \big)$ for (PA), we obtain ${\rm (PA')}$. Therefore, when $\bz^*$ is an optimal solution for ${\rm (PA')}$ and $x^*_v=\bar{\xi}_v^{-1}(\sum\nolimits_{e \in \delta(v)} z^*_e)$ for all $v \in V$, $(\bx^*,\bz^*)$ is an optimal solution for (PA). $\Box$

\bibliographystyle{abbrvnat}
\bibliography{reference}

\end{document}

%% file: head/package.tex
\usepackage[margin=1.0truein]{geometry}
\usepackage{setspace}
\usepackage{lscape}
\usepackage{pdflscape}

\usepackage{xcolor}

\usepackage{amsmath,amssymb,amsthm} 

\usepackage{wrapfig} 

\usepackage{mathrsfs} 
\usepackage{url} 
\usepackage{latexsym} 
\usepackage{bm} 

\usepackage{mathtools}

\usepackage{booktabs,multirow}
\usepackage{array}
\newcolumntype{P}[1]{>{\centering\arraybackslash}p{#1}}
\usepackage{diagbox}
\usepackage{threeparttable}
\usepackage{siunitx} 

\usepackage{algorithm}
\usepackage[noend]{algpseudocode}

\algnewcommand{\Break}{\textbf{break}}
\algnewcommand{\algorithmicgoto}{\textbf{go to} Line}
\algnewcommand{\Goto}[1]{\algorithmicgoto~\ref{#1}}
\algblockdefx{MRepeat}{EndRepeat}{\textbf{repeat}}{}
\algnotext{EndRepeat}

\usepackage{natbib}

\usepackage{thmtools,thm-restate}

\usepackage{hyperref}
\usepackage{cleveref}
\expandafter\def\csname ver@etex.sty\endcsname{3000/12/31} 

\usepackage{breakurl}

\usepackage{multicol}
\usepackage{wasysym} 
\usepackage{adjustbox}
\usepackage{xparse} 
\usepackage{pbox} 
\usepackage{authblk} 
\usepackage{xspace} 

\usepackage[subrefformat=parens]{subcaption}
\usepackage{enumitem} 
\hypersetup{
  colorlinks,
  linkcolor={green!35!black},
  citecolor={green!35!black},
  urlcolor={blue!50!black}
}

\usepackage{times}
\usepackage{soul}
\usepackage{secdot}
\usepackage{cases}
\usepackage{subcaption}
\usepackage{caption}

%% file: head/thm.tex
\declaretheoremstyle[
  shaded={bgcolor=gray!15},
]{thmsty}

\declaretheorem[
  name=Theorem,
  refname={Theorem,Theorems},
  style=thmsty,
]{theorem}

\declaretheorem[
  name=Lemma,
  refname={Lemma,Lemmas},
  style=thmsty,
  numberlike=theorem,
]{lemma}

\declaretheorem[
  name=Assumption,
  refname={Assumption,Assumptions},
  style=thmsty,
]{assumption}

\crefname{algorithm}{Algorithm}{Algorithms}
\crefname{line}{Line}{Lines}
\crefname{section}{Section}{Sections}
\crefname{appendix}{Appendix}{Appendices}
\crefname{table}{Table}{Tables}
\crefname{figure}{Figure}{Figures}
\crefname{equation}{}{}
\Crefname{equation}{Eq.}{Eqs.}

\setlist[itemize]{
  topsep=0.4\baselineskip,
  itemsep=0\baselineskip,
  leftmargin=1.5em,
}

\setlist[enumerate]{
  font=\upshape,
  label=(\alph*),
  ref=(\alph*),
  topsep=0.4\baselineskip,
  itemsep=0\baselineskip,
  leftmargin=2em,
}

\newlist{enuminasm}{enumerate}{1} 
\setlist[enuminasm]{
  font=\upshape,
  label=(\alph*),
  ref=\theassumption(\alph*),
  topsep=0.4\baselineskip,
  itemsep=0\baselineskip,
  leftmargin=2em,
}
\crefalias{enuminasmi}{assumption}

\newlist{enuminthm}{enumerate}{1}
\setlist[enuminthm]{
  font=\upshape,
  label=(\alph*),
  ref=\thetheorem(\alph*),
  topsep=0.4\baselineskip,
  itemsep=0\baselineskip,
  leftmargin=2em,
}
\crefalias{enuminthmi}{theorem}

\newlist{enuminlem}{enumerate}{1}
\setlist[enuminlem]{
  font=\upshape,
  label=(\alph*),
  ref=\thelemma(\alph*),
  topsep=0.4\baselineskip,
  itemsep=0\baselineskip,
  leftmargin=2em,
}
\crefalias{enuminlemi}{lemma}

\newlist{enuminprop}{enumerate}{1}
\setlist[enuminprop]{
  font=\upshape,
  label=(\alph*),
  ref=\theproposition(\alph*),
  topsep=0.4\baselineskip,
  itemsep=0\baselineskip,
  leftmargin=2em,
}
\crefalias{enuminpropi}{proposition}

\newlist{enumincond}{enumerate}{1}
\setlist[enumincond]{
  font=\upshape,
  label=(\alph*),
  ref=\thecondition(\alph*),
  topsep=0.4\baselineskip,
  itemsep=0\baselineskip,
  leftmargin=2em,
}
\crefalias{enumincondi}{condition}

%% file: head/macro.tex
\DeclareMathOperator{\Po}{Po}

\DeclareMathOperator*{\tsum}{\textstyle\sum}

\DeclareMathOperator{\EE}{\mathbb{E}}

\DeclarePairedDelimiterX\inner[2]{\langle}{\rangle}{{#1},{#2}}
\DeclarePairedDelimiter\abs{|}{|}

\DeclarePairedDelimiter\prn{(}{)}
\DeclarePairedDelimiter\bra{[}{]}
\DeclarePairedDelimiterX\Set[2]{\{}{\}}{\mspace{2mu}{#1}\;\delimsize|\;{#2}\mspace{2mu}}
\DeclarePairedDelimiterX\Prn[2]{(}{)}{\mspace{2mu}{#1}\;\delimsize|\;{#2}\mspace{2mu}}
\DeclarePairedDelimiterX\Bra[2]{[}{]}{\mspace{2mu}{#1}\;\delimsize|\;{#2}\mspace{2mu}}

\newcommand{\Z}{\mathbb Z}
\newcommand{\R}{\mathbb R}

\newcommand{\bx}{\bm{x}}

\newcommand{\bz}{\bm{z}}

\newcommand{\bH}{\bm{H}}

\newcommand{\bh}{\bm{h}}

\newcommand{\by}{\bm{y}}
\newcommand{\bmxi}{\bm{\xi}}

\newcommand{\mR}{\mathbb{R}}
\newcommand{\bmu}{\bm{\mu}}
\newcommand{\bY}{\bm{Y}}

\newcommand{\hbx}{\hat{\bx}}
\newcommand{\hbz}{\hat{\bz}}
\newcommand{\hx}{\hat{x}}

\newcommand{\E}{\mathbb{E}}

\newcommand{\Int}{\mathrm{Int}}

\renewcommand{\epsilon}{\varepsilon}

\NewDocumentCommand{\exsub}{s m O{} m}{%
  \IfBooleanT{#1}{\EE_{#2}\nolimits\bra*{#4}}%
  \IfBooleanF{#1}{\EE_{#2}\nolimits\bra[#3]{#4}}%
}
\NewDocumentCommand{\ex}{s O{} m}{%
  \IfBooleanT{#1}{\EE\nolimits\bra*{#3}}%
  \IfBooleanF{#1}{\EE\nolimits\bra[#2]{#3}}%
}
\NewDocumentCommand{\cex}{s O{} m m}{%
  \IfBooleanT{#1}{\EE\nolimits\Bra*{#3}{#4}}%
  \IfBooleanF{#1}{\EE\nolimits\Bra[#2]{#3}{#4}}%
}

\usepackage{CJKutf8}

\newcommand{\email}[1]{\href{mailto:#1}{\nolinkurl{#1}}}